%% file: main.tex
\documentclass[a4paper,11pt]{article}

\usepackage[headinclude]{typearea}
\usepackage[english]{babel}           
\usepackage[utf8]{inputenc}
\usepackage[T1]{fontenc}
\usepackage{amsmath,amsthm,amsfonts,amssymb}
\usepackage {color}                                  
\usepackage {pifont}                                 
\usepackage {multicol}                               
\usepackage{natbib}                                                     
\usepackage {courier}                                
\usepackage {graphicx}                               
\usepackage {array}                                  
\usepackage {longtable}                              
\usepackage{fancyvrb}
\usepackage{enumerate} 
\usepackage{hyperref}
\usepackage{ifpdf}
\usepackage{caption}

\usepackage{setspace}

\usepackage{marvosym}

\allowdisplaybreaks

\theoremstyle{plain}
\newtheorem{theorem}{Theorem}[section]

\newtheorem{lemma}[theorem]{Lemma}
\newtheorem{proposition}[theorem]{Proposition}

\theoremstyle{definition}
\newtheorem{assumption}[theorem]{Assumption}
\newtheorem{definition}[theorem]{Definition}
\newtheorem{example}[theorem]{Example}
\newtheorem{remark}[theorem]{Remark}

\newcommand{\abs}[1]{\left\vert#1\right\vert}
\newcommand{\E}{{\mathbb{E}}}
\newcommand{\N}{{\mathbb{N}}}
\renewcommand{\P}{{\mathbb{P}}}

\newcommand{\R}{{\mathbb{R}}}

\newcommand{\calX}{\mathcal{X}}

\newcommand{\Qsmooth}{Q^\varepsilon}
\newcommand{\odecoeff}{g}

\newcommand{\xsol}{\kappa}

\definecolor{lightred}{rgb}{1.0,0.6,0.6}
\definecolor{darkgreen}{rgb}{0,0.5,0}

\definecolor{magenta}{rgb}{0.75,0,0.25}

\definecolor{violet}{rgb}{0.25,0,0.75}

\definecolor{col3}{rgb}{0,0.25,0.75}

\definecolor{col4}{rgb}{0,0.75,0.25}

\newcommand{\veps}{\varepsilon}

\newcommand{\sheavi}{h}

\newcommand{\1}{{\mathbf 1}}
\renewcommand{\P}{{\mathbb P}}

\newcommand{\cA}{{\cal A}}
\newcommand{\cB}{{\cal B}}

\newcommand{\cD}{{\cal D}}
\newcommand{\cE}{{\cal E}}
\newcommand{\cF}{{\cal F}}

\newcommand{\cH}{{\cal H}}

\newcommand{\be}{\begin{equation}}
\newcommand{\ee}{\end{equation}}
\newcommand{\bea}{\begin{eqnarray}}
\newcommand{\eea}{\end{eqnarray}}

\newcommand{\bsx}{{\boldsymbol{x}}}
\newcommand{\norm}[1]{\left\Vert#1\right\Vert}

\title{Approximation methods for piecewise deterministic Markov processes and their costs}

\author{Peter Kritzer\thanks{
P.~Kritzer is supported by the Austrian Science Fund (FWF): Project F5506-N26, which is part of the Special Research Program `Quasi-Monte Carlo Methods: Theory and Applications'.
P.~Kritzer is partially supported by the National Science Foundation (NSF) under Grant DMS-1638521 to the	
Statistical and Applied Mathematical Sciences Institute. Any opinions, findings, and conclusions or recommendations expressed	
in this material are those of the authors and do not necessarily reflect the views of the National Science Foundation.}\;\,\thanks{P. Kritzer, G. Leobacher, and M.~Sz\"olgyenyi gratefully acknowledge the partial support of the Erwin Schr\"odinger International Institute for Mathematics and Physics (ESI) in Vienna under the thematic programme `Tractability of High Dimensional Problems and Discrepancy'}
 \and Gunther Leobacher\thanks{G.~Leobacher is supported by the Austrian
Science Fund (FWF): Project F5508-N26, which is part of the Special Research
Program `Quasi-Monte Carlo Methods: Theory and Applications'. Part of this
article was written while G.~Leobacher was affiliated with the Institute of
Financial Mathematics and Applied Number Theory, Johannes Kepler University
Linz, Altenbergerstraße 69, 4040 Linz, Austria.  } \footnotemark[2]\and Michaela Sz\"olgyenyi\thanks{ M.~Sz\"olgyenyi is supported by the AXA Research Fund grant `Numerical Methods for Stochastic Differential Equations with Irregular Coefficients with Applications in Risk Theory and Mathematical Finance'.
A part of this article was written while M.~Sz\"olgyenyi was affiliated with the Institute of Statistics and Mathematics, Vienna University of Economics and Business, Welthandelsplatz 1, 1020 Vienna, Austria, and supported by the Vienna Science and Technology Fund (WWTF): Project MA14-031.
} \footnotemark[2] \and Stefan Thonhauser}
\begin{document}

\date{Preprint 2018}

\maketitle


\begin{abstract}%
In this paper, we analyse piecewise deterministic Markov processes, as introduced 
in \citet{Davis1984}. 
Many  models in insurance mathematics can be formulated in terms
of the general concept of piecewise deterministic Markov processes.
In this context, one is interested in computing certain quantities of interest such as the probability of ruin of an insurance company, or the insurance company's value, defined as the expected discounted future dividend payments until the time of ruin. 
Instead of explicitly solving the integro-(partial) differential 
equation related to the quantity of interest considered (an approach which can only be used in few special cases),
we adapt the problem in a manner that allows us to apply deterministic numerical integration algorithms such 
as quasi-Monte Carlo rules; this is in contrast to applying random integration algorithms such as Monte Carlo. 
To this end, we reformulate a general cost functional as a fixed point of a particular integral operator, which allows for iterative
approximation of the functional. Furthermore, we introduce a smoothing technique which is applied to the 
integrands involved, in order to use error bounds for deterministic cubature rules. 
On the analytical side, we prove a convergence result for our PDMP approximation,
which is of independent interest as it justifies phase-type approximations on the process level.
We illustrate the smoothing technique for a risk-theoretic example, and provide a
comparative study of deterministic and Monte Carlo
integration.\\

\end{abstract}

\section{Introduction}

Many models in risk theory can be formulated as piecewise deterministic Markov processes (PDMPs)---a general class of finite-variation sample path Markov processes introduced by \citet{Davis1984}. 
This applies, among others, to the classical Cram\'er-Lundberg model, the renewal risk models, and multi-portfolio models 
recently introduced by \citet{albrecherLautscham2015}.
Moreover, PDMPs are sufficiently general to allow for non-constant model parameters, i.e., quantities such as 
the hazard rate or the premium rate may be state dependent.
Examples of PDMPs and their control in the field of insurance mathematics are, e.g., \citet{DassiosEmbrechts1989},
 \citet{EmbrechtsSchmidli1994,schal,Rolski1999,Cai2009,LeoNgare,sz2016d}.
 
The general theory of PDMPs is well  developed, see for example the monographs by \citet{davis1993}, \citet{Jacobsen2006}, or
\citet{bauerle2011} for general results on PDMPs and their optimal control.
More specialised contributions to the control theory of PDMPs can be found in 
\citet{davis1993, lenhart1985integro, costa1989impulse, dempster1992necessary, almudevar2001dynamic, forwick2004piecewise,bauerle2010optimal,costa2013continuous}, or
\citet{farid1999} for viscosity solutions of associated Hamilton-Jacobi-Bellman equations, and \citet{sz2016c} for a general comparison principle for solutions to control problems for PDMPs.

For the numerical treatment of (control) problems for PDMPs, however, only problem-specific solutions have been provided.
A standard approach is to link expected values representing a quantity of interest in the problem to the solution of an associated
integro-(partial) differential equation, see, e.g.,  \citet{asmussen2010}. In only very few cases is it possible to 
derive an explicit solution to this integro-(partial) differential equation. Requiring an explicit solution typically restricts the  complexity of the model significantly. 
One possibility is to solve the integro-(partial) differential equation numerically. 
This carries all the intricacies and difficulties of a combined numerical method for differential and integral equations.
Alternatively one can apply crude Monte Carlo methods, see, e.g., \citet{Riedler}.
Those methods, while robust, are limited in speed by the Monte Carlo convergence rate.
Another---highly sophisticated---approach uses quantisation of the jump distribution, see \citet{Dufour2016}.

In this article we concentrate on particularly easy to implement methods similar to Monte Carlo.
The aim is to adapt the problem in a way that also allows for 
deterministic numerical integration algorithms such as quasi-Monte Carlo (QMC).
QMC has been applied successfully to problems in risk theory, see
\citet{Tichy1984,Lefevre2008,Siegl2000,AlbrecherKainhofer2002,preischl2018}.
It should be noted that the finiteness of  the total variation needed 
for the convergence estimate \cite[Theorem 1]{AlbrecherKainhofer2002} 
has not been proven.

We would like to highlight two features of our approach. Inspired by \cite{AlbrecherKainhofer2002}, 
we reformulate a general cost functional as a fixed point of a particular integral operator,
which allows for iterative approximation of the functional. In terms of numerical integration this means that we get a high-dimensional
integration problem of fixed dimension, where the dimension is a multiple of the number of iterations. Having a fixed  dimension is required 
for the application of standard QMC or other deterministic cubature rules. 

The application of QMC requires some degree of regularity of the integrand. Only in rare cases these will be satisfied automatically. 
The examples from risk theory considered here lead to non-smooth integrands. For these situations, we introduce a
smoothing technique which, in its simplest case, leads to $C^2$ integrands. From the earlier considerations, we obtain deterministic error bounds 
for those. We prove convergence in distribution of the ``smoothed processes'' to the original ones, which implies convergence of the 
corresponding expected values for every initial value of the process. In Section \ref{sec:examples} we even obtain uniform convergence 
with respect to the initial value in a particular setup from risk theory.

Our convergence result has an additional benefit for a typical situation in risk-theoretic modelling. In the literature on the analysis of ruin probabilities, or more generally, on  Gerber-
Shiu functions, the assumption of a claim size distribution of mixed exponential or phase-type form is quite common. Apart from the possibility to obtain explicit expressions for quantities of interest 
in such setups, this modelling approach is motivated by the fact that the class of phase-type distributions is dense in the class of distributions with support on $[0,\infty)$, see 
\cite[Theorem 8.2.3]{Rolski1999}. Under mild assumptions on the claim size distribution we want to approximate, our convergence result applies 
and justifies the phase-type approximation procedure even on the process level.
Furthermore, efficient and easy to implement numerical methods for the computation of important targets such as Gerber-Shiu functions and expected discounted future dividend payments of an insurance 
company are of particular importance when models become more general and hence also more complicated.
This makes our contribution valuable from both the analytical and the numerical point of view.

We would like to emphasize that the methods presented here per se do not provide solutions to optimal 
control problems, which is the main application of PDMPs in risk theory. However, the integration algorithms as introduced here can be used 
in a policy iteration procedure for calculating costs associated with a fixed policy.

The paper is structured as follows. In Section \ref{sec:PDMPs} we recall the definition of a PDMP and provide some risk-theoretic examples.
In Section \ref{sec:fixed-point-approach} we derive the fixed point approach for valuation of a cost functional of a PDMP.
Section \ref{sec:cubature} reviews deterministic numerical integration of possibly multivariate $C^k$ functions. 
Subsequently, Section \ref{sec:smoothing} is devoted to the aforementioned smoothing procedure, and presents a stability result.
Section \ref{sec:DE} contains an application of the smoothing to one of the risk-theoretic examples and a comparative study of deterministic and Monte Carlo integration for this example. \\

\section{Piecewise deterministic Markov processes}\label{sec:PDMPs}

In this section we first define piecewise deterministic Markov processes. Then we give a couple of examples of practical interest.

A PDMP is a continuous-time stochastic process with (possibly random) jumps,
which follows a deterministic flow, e.g., the solution of an ordinary
differential equation, between jump times.  
We will not give the most general
definition of PDMPs here, but instead 
refer to the monograph by \citet{davis1993}.
For a subset $A$ of $\R^d$ we denote by $A^\circ,\bar A$, and $\partial A$ its
interior, closure, and boundary, respectively.
We write $\cB(A)$ for the Borel $\sigma$-algebra on $A$.

\begin{definition}
Let $A\subseteq \R^d$.
A function $\phi\colon A\times \R\to \R^d$ is called a {\em flow} on $A$, if
\begin{itemize}
\item  $\phi$ is
continuous, 
\item $\phi(x,0)=x$ for all $x\in A$;
\item for all $x\in A$ and all $s,\,t\in\R$ it holds that 
if $\phi(x,t)\in A$ and $\phi(\phi(x,t),s)\in A$ 
then $\phi(x,t+s)=\phi(\phi(x,t),s)$.
\end{itemize}
For fixed $x\in A$, let $\phi^{-1}(x,A)=\{t\in\R\colon \phi(x,t)\in A\}$. 
Then the function $\phi(x,\cdot)\colon \phi^{-1}(x,A) \to A$ is called a {\em trajectory} of the flow.
\end{definition}

If $\phi$ is a flow on $A$, then we write 
$\partial^-_\phi A=\{x\in \partial A\colon \exists \varepsilon\in(0,\infty)\ \mbox{such that}\
\forall t\in (0,\varepsilon)\colon  \phi(x,t)\in A^\circ\}$ and 
$\partial^+_\phi A=\{x\in \partial A\colon \exists \varepsilon\in(0,\infty)\ \mbox{such that}\ 
\forall t\in (0,\varepsilon)\colon  \phi(x,-t)\in A^\circ\}$.  

Thus $\partial^-_\phi A$ consists of the points on the boundary of $A$ 
from which the
trajectory moves into $A^\circ$ immediately, and $\partial^+_\phi A$ 
consists of the points on the boundary of $A$ 
to which a 
trajectory moves from $A^\circ$ without passing other points on the
boundary in-between.  Furthermore, we write $\partial^1_\phi A\colon
=\partial^-_\phi A\backslash \partial^+_\phi A$.

\begin{remark}
The classical example of a flow arises through ordinary differential equations
(ODEs). Let $\odecoeff:\R^d\to \R^d$ be Lipschitz continuous. 
By the classical Picard-Lindelöf theorem on existence and uniqueness 
of solutions of ODEs we have that for every $x\in
\R$ there exists a continuously 
differentiable function 
$\xsol:\R\to \R^d$ such that $\xsol(0)=x$ and $\xsol'(s)=\odecoeff(\xsol(s))$ for 
all $s\in \R$. For $t\in \R$ we define $\phi(x,t)=\xsol(t)$. The function
$\phi$ defines a flow on $\R^d$. If $A\subseteq \R^d$, then the restriction 
of $\phi$ to $A\times \R$ is a flow on $A$.
\end{remark}

\begin{definition}\label{def:PDMP-data}
Let $K$ be a finite set and let $d\colon K\to\N$ be a function which satisfies that, for every $k\in K$, 
$E_k\subseteq\R^{d(k)}$ and $\phi_k$ is a flow on $E_k$ with $E_k=E_k^\circ\cup \partial^1_{\phi_k} E_k$.
\begin{enumerate}[(i)]
\item The \emph{state space} $(E,\cE)$ of a PDMP is the measurable space 
defined by $E=\bigcup_{k\in K}(\{k\}\times E_k)$ and $\cE=\sigma(\{\{k\}\times B\colon  k\in K, B\in \cB(E_k)\})$.
\item The \emph{flow} of a PDMP is defined by $\phi=\{\phi_k\}_{k\in K}$.
\item The {\em active boundary} of the PDMP is defined by
$\Gamma^\ast =\bigcup_{k=1}^K \partial^+_{\phi_k} E_k$. Furthermore, we define a 
$\sigma$-algebra on $E\cup \Gamma^*$ by
$\cE^\ast =\sigma(\{\{k\}\times B\colon  k\in K, B\in \cB(E_k\cup \partial^+_{\phi_k} E_k)\})$.
\item The \emph{jump intensity} $\lambda$ of a PDMP is defined by a family of functions $\lambda=\{\lambda_k\}_{k\in K}$ with
$\lambda_k\colon E_k\to [0,\infty)$ measurable and bounded for all $k\in K$.
\item The {\em jump kernel} $Q$ of a PDMP is a function $Q\colon \mathcal{E}\times (E\cup\Gamma^\ast )\to[0,1]$ such that 
 $Q(A,\cdot)$ is $\mathcal{E}^\ast $-$\cB([0,1])$
measurable for every $A\in\mathcal{E}$, and $Q(\cdot,x)$ is a 
probability measure
on $(E,\mathcal{E})$ for every $x\in E$ with $Q(\{x\},x)=0$.
\end{enumerate}
We call the triple $(\phi,\lambda,Q)$ the {\em local characteristics}
of a PDMP.
\end{definition}

Given a state space $(E,\cE)$ and local characteristics $(\phi,\lambda,Q)$ of a PDMP we define
the function
$t^\ast \colon E\to [0,\infty]$ by 
\begin{align*}
t^\ast (k,y)&=\begin{cases}
\inf\{t>0\,\colon \,\phi_k(y,t)\in\partial^+_{\phi_k} E_k\} & \text{ if } \exists 
t>0\colon \phi_k(y,t)\in\partial^+_{\phi_k} E_k,\\
\infty & \text{ otherwise.}
\end{cases}
\end{align*}

\begin{definition}\label{def:PDMP}
Let $(E,\cE)$ be a state space and let $(\phi,\lambda,Q)$ be local
characteristics  of a PDMP, let $x\in E$, and let $(\Omega,\cF,\P)$ be a probability
space.  A \emph{piecewise deterministic Markov process} starting in $x$
is a stochastic process $X\colon[0,\infty)\times\Omega\to E$ which satisfies the following. 
There exists a sequence of random variables $(T_n)_{n\in\N}$ 
with $T_n\in [0,\infty]$ and $T_n\le T_{n+1}$ a.s.~and $\lim_{n\to \infty} T_n=\infty$ a.s.~for all $n\in\N$ such that
\begin{enumerate}[(i)]
\item it holds $\P$-a.s.~that $X_0=x$,
\item for all $n\in\N$, $t\in [T_n,T_{n+1})$, and for $(k,y)\in E$ with $X_{T_n}=(k,y)$ it holds $\P$-a.s.~that
$X_t=\phi_k(y,t-T_n)$, 
\item for all $s,t\in [0,\infty)$ it holds $\P$-a.s.~that
\[
\P\big(T_{n+1}-T_n>t|X_s=(k,y)\text{ and } T_n\le s<T_{n+1}\big)=
\begin{cases}
e^{-\int_0^t\lambda_k(\phi_k(y,\tau))d\tau }&\mbox{if}\ 0<t<t^\ast (k,y),\\
0&\mbox{if}\ t\geq t^\ast (k,y),
\end{cases}
\]
\item for all $n\in\N$ and all $A\in \cE$ it holds $\P$-a.s.~that
\[
\P\big(X_{T_{n+1}}\in A|X_{T_n-}\big)=Q(A,X_{T_n}).
\]
\end{enumerate}
\end{definition}

\begin{theorem}\label{th:existence-pdmp}
Let $(E,\cE)$ be a state space and let $(\phi,\lambda,Q)$ be local characteristics  of a PDMP, let $x\in E$.
There exist a probability space $(\Omega,\cF,\P_{x})$  
and a stochastic process
$X\colon[0,\infty)\times\Omega\to E$ such that
$X$ is a PDMP starting in $x$ with state space $E$ and local characteristics $(\phi,\lambda,Q)$. 
Furthermore, $X$ has the strong Markov property.
\end{theorem}

\begin{proof}
The proof of Theorem \ref{th:existence-pdmp} for a more general setup
that also allows for the possibility of explosions and countable $K$ 
can be found in \cite[Section 2.25]{davis1993}.
\end{proof}

Figure \ref{fig:PDMP} illustrates a path of a PDMP. \\
\begin{figure}
\begin{center}
\input{pdmp_gen_klein7.pspdftex}
\end{center}
\caption{Illustration of a PDMP.}\label{fig:PDMP}
\end{figure}
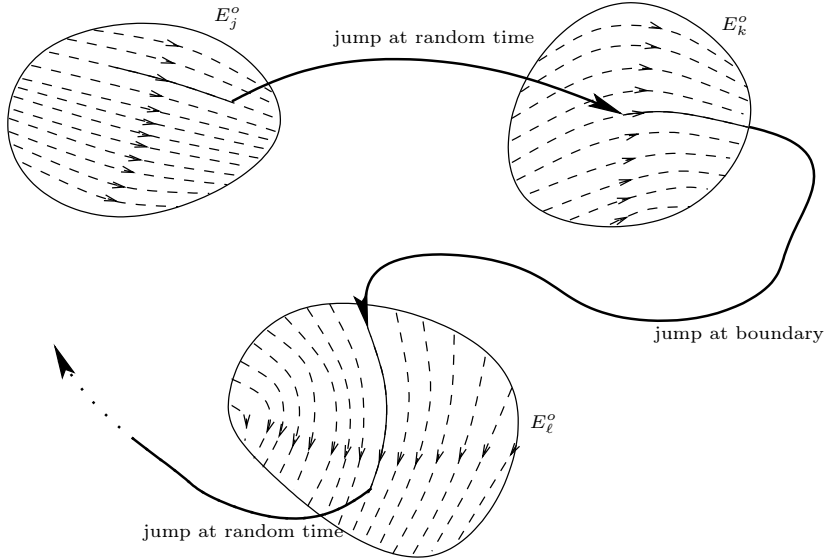

Let $f\colon E\to\R$ be a function.
For all $k\in K$ we denote by $f_k$ the function $f_k\colon E_k\to \R$ which satisfies for all $x\in E_k$ that $f_k(x)= f(k,x)$.
It is not hard to see that $f$ is measurable if and only if $f_k$ is measurable
for every $k\in K$.
We say that $f$ is $n$-times continuously differentiable, if
for every $k\in K$  there exists 
an open set $A_k\subseteq\R^{d(k)}$ with $E_k\subseteq A_k$ and an 
$n$-times continuously differentiable function
$\hat f_k:A_k\to \R$ such that $f_k=\hat f_k|_{E_k}$.
We write $C^n(E,\R^m)$ for the space of $n$-times differentiable functions on 
$E$ and $C^n_b(E,\R^m)$ for the space of functions in $C^n(E,\R^m)$ for which all
derivatives are bounded. Moreover, $C^n_0(E,\R^m)$ is the space of functions  
in $C^n_b(E,\R^m)$ for which all derivatives vanish at infinity.

Further, for $f\colon E\to\R$, a PDMP $X$, and
$t\in(0,\infty)$ we write $\E(f(X_t)|X_0=x)=:\E_x(f(X_t))$.

In the remainder of this section we provide some illustrative examples from risk theory.
For other examples and applications in different fields we refer to \citet{davis1993,Dufour2012,Riedler}.

\subsection{Examples} \label{sec:examples}
\subsubsection{Classical Cram\'er-Lundberg model}
\label{subsubsec:C-L}
Let $X=(X_t)_{t\geq 0}$ be a stochastic process given by
\begin{align}\label{eq:Surplus}
X_t=x+c\,t-S_t,\quad t\geq 0,
\end{align}
where $x,\,c\,\geq 0$, $N=(N_t)_{t\geq 0}$ is a homogeneous Poisson process with intensity $\lambda_N>0$,
$\{Y_i\}_{i\in\N}$ is a family of positive i.i.d.~random variables with distribution function $F_Y$, and $S_t=\sum_{i=1}^{N_t}Y_i$ for all $t\ge 0$.
A usual assumption in this kind of model is the independence of $\{Y_i\}_{i\in\N}$ and $N$.
In risk theory the process $X$ represents a standard model for the surplus of an insurance portfolio.
A quantity of interest is the probability of $X$ ever becoming negative,
i.e., we are interested in 
\(
\P(\tau<\infty)\,,
\)
where $\tau=\inf\{t\ge 0: X_t<0\}$.
The model translates into a PDMP  via
\begin{itemize}
\setlength{\itemsep}{0em}
\item $K=\{1,2\}$,
\item  $E_1=[0,\infty)$, $E_2=(-\infty,0)$, 
\item $\phi_1(y,t)=y+ct$ 
$\forall y\in E_1$ and $\forall t\in \R$, 
$\phi_2(y,t)=y$ $\forall y\in E_2$ and $\forall t\in \R$, 
\item $\lambda_1(y)=\lambda_N$ $\forall y\in E_1$, 
$\lambda_2(y)=0$ $\forall y\in E_2$.
\item For $B_1\in \cB (E_1)$, $B_2\in \cB (E_2)$, and  $B=(\{1\}\times B_1)\cup(\{2\}\times B_2)$, 
\begin{align*}
Q(B,(1,y))=\P(Y\in y-B_1)+\P(Y\in y-B_2)
\end{align*}
for $y\in E_1$, and 
$Q(B,(2,y))=\P(Y\in y-B_2)$,
\end{itemize}
where we have used the notation $y-B=\{y-y' \colon y'\in B\}$
for all  $y\in \R$ and $B\in\mathcal{B}(\R)$.
For $y\in E_2$, any definition for $Q$ will do, since the jump intensity 
is 0 there, but the above definition is provided for definiteness.

\subsubsection{Cram\'er-Lundberg model with dividend payments}
\label{subsubsec:C-L-div}
A classical modification of the model from Section \ref{subsubsec:C-L} is the introduction of a dividend barrier at level $b>0$.
Then, once the surplus reaches the barrier, the incoming premium rate is immediately distributed as a dividend.
Furthermore, if the process starts above $b$, the excess is distributed as a lump sum dividend, such that
$X_{0+}=\min\{x,b\}$.
A typical quantity of interest is the expected value of discounted future dividend payouts until ruin of the company,
which is given by 
\begin{equation}\label{eq:val-barrier}
V(x )=
\begin{cases}
\E_{x} \left(\int_0^\tau e^{-\delta t} c \1_{\{X_t = b\}}dt\right)&\mbox{if}\ x  \le b,\\
x -b+ \E_b\left(\int_0^\tau e^{-\delta t} c \1_{\{X_t = b\}}dt\right)&\mbox{if}\ x  > b,
\end{cases}
\end{equation}
where $\delta >0$ is a preference-based discount factor and $\tau=\inf\{t\geq 0 \colon X_t<0\}$.
The model translates into a PDMP via
\begin{itemize}
\setlength{\itemsep}{0em}
\item $K=\{1,2,3\}$, 
\item $E_1=[0,b)$, $E_2=(-\infty,0)$, $E_3=\{b\}$, 
\item $\phi_1(y,t)=y+ct$ 
$\forall y\in E_1$ and $\forall t\in \R$, 
$\phi_2(y,t)=y$ $\forall y\in E_2$ and $\forall t\in \R$, 
$\phi_3(y,t)=y$ $\forall y\in E_3$ and $\forall t\in\R$, 
\item $\lambda_1(y)=\lambda_N$ $\forall y\in E_1$, 
$\lambda_2(y)=0$ $\forall y\in E_2$, 
$\lambda_3(y)=\lambda_N$ $\forall y\in E_3$. 
\item For $B_k\in \cB (E_k)$, $1\le k\le 3$, and $B=(\{1\}\times B_1)\cup(\{2\}\times B_2)\cup(\{3\}\times B_3)$, 
\begin{align*}
Q(B,(1,y))=\P(Y\in y-B_1)+\P(Y\in y-B_2)
\end{align*}
for $y\in E_1$,  
$Q(B,(2,y))=\P(Y\in y-B_2)$ 
for $y\in E_2$, and
\begin{align*}
Q(B,(3,y))=\P(Y\in y-B_1)+\P(Y\in y-B_2)
\end{align*}
for $y\in E_3$. Finally, $Q(B,(1,y))=\1_{B_3}(y)(3,y)$ for 
$y\in \partial^1_{\phi_1}E_1=\{b\}$.
\end{itemize}
Note that only initial values $x\in (-\infty,b]$ translate 
to a viable initial value for the PDMP. However, this is sufficient for 
determining $V(x)$ for all $x\in \R$ via \eqref{eq:val-barrier}.\\

\subsubsection{Cram\'er-Lundberg model with time dependent dividend barrier}
\label{subsubsec:C-L-div-timedep}
In \citet{AlbrecherKainhofer2002} the model from Section \ref{subsubsec:C-L-div} is further extended to include a time dependent barrier $b\colon[0,\infty)\to|0,\infty)$ of the form
\begin{align*}
b(t)=\left(b_0^m+\frac{t}{\alpha}\right)^{\frac 1 m},
\end{align*}
where $ \alpha,b_0>0$, $m>1$.
The quantity of interest is again the expected value of discounted future dividend payments until the time of ruin, i.e.,
\begin{align*}
V(x)=\E_x\left(\int_0^\tau e^{-\delta t} (c-b_t)\1_{\{X_t=b_t\}}dt\right),
\end{align*}
for $x\le b_0$, where again $\tau=\inf\{t\geq 0 \colon X_t<0\}$ and $\delta >0$ is a preference-based discount factor.
The model translates into a PDMP via
\begin{itemize}
\setlength{\itemsep}{0em}
\item $K=\{1,2,3\}$, 
\item $E_1=\{(s,y)\in \R^2:0\le y< b(s)\}$, $E_2=\{(s,y)\in \R^2: y<0 \}$, $E_3=\{(s,y)\in \R^2:y= b(s)\}$, 
\item $\phi_1((s,y),t)=(s+t,y+ct)$ 
$\forall (s,y)\in E_1$ and $\forall t\in \R$, 
$\phi_2((s,y),t)=(s+t,y)$ $\forall y\in E_2$ and $\forall t\in\R$, 
$\phi_3((s,y),t)=(s+t,b(s+t))$ $\forall (s,y)\in E_3$ and $\forall t\in\R$, 
\item $\lambda_1(y)=\lambda_N$ $\forall y\in E_1$, 
$\lambda_2(y)=0$ $\forall y\in E_2$, 
$\lambda_3(y)=\lambda_N$ $\forall y\in E_3$.
\item For $B_k\in \cB (E_k)$, $1\le k\le 3$, and $B=(\{1\}\times B_1)\cup(\{2\}\times B_2)\cup(\{3\}\times B_3)$, 
\begin{align*}
Q(B,(1,(s,y)))=\P(Y\in y-(\{s\}\times \R)\cap B_1)+\P(Y\in y-(\{s\}\times \R)\cap B_2)
\end{align*}
for $(s,y)\in E_1$,  
$Q(B,(2,(s,y)))=\P(Y\in y-(\{s\}\times \R)\cap B_2)$
for $(s,y)\in E_2$, and
\begin{align*}
Q(B,(3,(s,y)))=\P(Y\in y-(\{s\}\times \R)\cap B_1)+\P(Y\in y-(\{s\}\times \R)\cap B_2)
\end{align*}
for $(s,y)\in E_3$. Finally, $Q(B,(1,(s,y)))=\1_{B_3}((s,y))(3,(s,y))$ for 
$(s,y)\in \partial^1_{\phi_1}E_1=E_3$.
\end{itemize}

\subsubsection{Cram\'er-Lundberg model with loan}\label{subsubsec:CL-loan}
In \citet{DassiosEmbrechts1989} the model from Section \ref{subsubsec:C-L-div} is modified such that
the insurance company is not ruined when the surplus hits zero, but has the possibility to take up a loan at an interest rate $\rho>0$.
The time of ruin is given by $\tau=\inf\{t\geq 0 \colon X_t < -c/\rho\}$.
The corresponding quantity of interest is
\begin{align*}
V(x)=\E_x\left(\int_0^\tau e^{-\delta t} c \1_{\{X_t = b\}}dt\right),
\end{align*}
for $x\le b$, where $\delta >0$ is a preference-based discount factor.
The model translates into a PDMP via
\begin{itemize}
\setlength{\itemsep}{0em}
\item $K=\{1,2,3,4,5\}$, 
\item $E_1=[0,b)$, $E_2=(-\frac{c}{\rho},0)$, $E_3=\{b\}$, $E_4=(-\infty,-\frac{c}{\rho})$, $E_5=\{-\frac{c}{\rho}\}$,
\item $\phi_1(y,t)=y+ct$ 
$\forall y\in E_1$ and $\forall t\in \R$, 
$\phi_2(y,t)=y$ $\forall y\in E_2$ and $\forall t\in\R$, 
$\phi_3$ is the flow of the ODE $z'=c+\rho z$ at $(y,t)$ 
$\forall y\in E_3$ and $\forall t\in \R$, 
$\phi_4(y,t)=y$ $\forall y\in E_4$ and  $\forall t\in \R$, 
$\phi_5(y,t)=y$ $\forall y\in E_5$ and $\forall t\in\R$, 
\item $\lambda_1(y)=\lambda_N$ $\forall y\in E_1$, 
$\lambda_2(y)=\lambda_N$ $\forall y\in E_2$, 
$\lambda_3(y)=\lambda_N$ $\forall y\in E_3$, 
$\lambda_4(y)=0$ $\forall y\in E_4$, 
$\lambda_5(y)=0$ $\forall y\in E_5$. 
\item For $B_k\in \cB (E_k)$, $1\le k\le 5$, and 
$B=\bigcup_{k=1}^5(\{k\}\times B_k)$, 
\begin{align*}
Q(B,(1,y))=\P(Y\in y-B_1)+\P(Y\in y-B_2)+\P(Y\in y-B_4)
\end{align*}
for $y\in E_1$,  
$Q(B,(2,y))=\P(Y\in y-B_2)+\P(Y\in y-B_4)$ 
for $y\in E_2$, and
\begin{align*}
Q(B,(3,y))=\P(Y\in y-B_1)+\P(Y\in y-B_2)
\end{align*}
for $y\in E_3$. Finally, $Q(B,(1,y))=\1_{B_3}(y)(3,y)$ for 
$y\in \partial^1_{\phi_1}E_1=\{b\}$, and $Q(B,(2,y))=\1_{B_2}(y)(1,y)$ for 
$y\in \partial^1_{\phi_2}E_2=\{0\}$.\\
\end{itemize}

\subsubsection{Multidimensional Cram\'er-Lundberg model}\label{subsubsec:C-L-multidim}
In \citet{albrecherLautscham2015} a two-dimensional
extension of the model in Section \ref{subsubsec:C-L-div} is studied.
The basis are independent surplus processes modelling two insurance portfolios 
$X^{(j)}_t=x^{(j)}+c^{(j)} t -S^{(j)}_t$, $j\in\{1,2\}$,
where $c^{(1)},c^{(2)}\ge 0$ and $S^{(j)}$ are compound Poisson processes with 
intensities $\lambda^{(1)},\,\lambda^{(2)}$ and jump size distributions
${F}_{Y^{(1)}},\,{F}_{Y^{(2)}}$.
Furthermore,  $b^{(1)},b^{(2)}\ge 0$ are barriers. As a new feature, the drift of the
component at the barrier is added to the other component's drift,
causing faster growth of the latter. Dividends are only paid when 
both surplus processes have reached their individual barriers.
We show how the model translates into a PDMP, namely
\begin{align*}
E_1&=\{(x^{(1)},x^{(2)})\in\R^2\,:\,0\leq x^{(1)}<b^{(1)},\,0\leq x^{(2)}<b^{(2)}\},\\
E_2&=\{(x^{(1)},x^{(2)})\in\R^2\,:\, b^{(1)}= x^{(1)},\,0\leq x^{(2)}<b^{(2)}\},\\
E_3&=\{(x^{(1)},x^{(2)})\in\R^2\,:\,0\leq x^{(1)}<b^{(1)},\,b^{(2)}=x^{(2)}\},\\
E_4&=\{(x^{(1)},x^{(2)})\in\R^2\,:\,b^{(1)}= x^{(1)},\,b^{(2)}= x^{(2)}\},\\
E_5&=\R^2\setminus (E_1\cup E_2\cup E_3\cup E_4).
\end{align*}
The flow is given by
\begin{align*}
\phi_1(x,t)=x+\begin{pmatrix}
c^{(1)}\\
c^{(2)}
\end{pmatrix}t\,,\qquad
\phi_2(x,t)=x+\begin{pmatrix}
0\\
c^{(1)}+c^{(2)}
\end{pmatrix}t\,,\qquad
\phi_3(x,t)=x+\begin{pmatrix}
c^{(1)}+c^{(2)}
\\0
\end{pmatrix}t\,,
\end{align*}
and $\phi_4(x,t)=\phi_5(x,t)=x$ for all $x\in \R^2$, $t\ge0$.
It remains to describe the jump behaviour. We get deterministic `jumps' 
at the active boundaries of $E_1,E_2,E_3$ which do not manifest themselves
as jumps of the process, i.e., $Q(A,(1,x))=\1_A((2,x))$ for 
$(1,x)\in \partial^1_{\phi_1}(E_1)$ and similar for the other active boundaries.
Since each surplus process is a compound Poisson process with drift, 
jumps in the
components occur due to realisations of independent identically distributed exponential
random variables (independence implies that mutual jumps occur with probability zero).
The two-dimensional process thus jumps at the minimum of the individual
jump times. This means that we have a constant jump intensity 
 $\lambda_k=\lambda^{(1)}+\lambda^{(2)}$ for $k=1,2,3,4$,  and 
$\lambda_5=0$.
If a jump occurs at time $t\ge0$, it happens with
probability $\frac{\lambda^{(1)}}{\lambda^{(1)}+\lambda^{(2)}}$ in the first surplus process with
jump size distribution ${F}_{Y^{(1)}}$, and with probability
$\frac{\lambda^{(2)}}{\lambda^{(1)}+\lambda^{(2)}}$ in the second surplus process with jump size
distribution ${F}_{Y^{(2)}}$. It remains to  
describe the jump kernel for the jumps from $x\in E$. 
To this end define, for $k_1,k_2\in\{1,2,3,4\}$ and 
$B\in\mathcal{B}(E_{k_2})\subseteq\mathcal{B}(\R^2)$, and 
$(y^{(1)},y^{(2)})\in E_{k_1}$, 
\begin{align*}
B^{(1)}&=\{z^{(1)}\in \R\colon(z^{(1)},z^{(2)})\in B, z^{(2)}=y^{(2)}\}\, ,\\
B^{(2)}&=\{z^{(2)}\in \R\colon(z^{(1)},z^{(2)})\in B, z^{(1)}=y^{(1)}\}\,.
\end{align*}
Furthermore,
\begin{align*}
Q(\{k_2\}\times B,(k_1,y^{(1)},y^{(2)}))
=\frac{\lambda^{(1)}}{\lambda^{(1)}+\lambda^{(2)}}{F}_{Y^{(1)}}(y^{(1)}-B^{(1)})
+\frac{\lambda^{(2)}}{\lambda^{(1)}+\lambda^{(2)}}{F}_{Y^{(2)}}(y^{(2)}-B^{(2)})\,.
\end{align*}

A quantity of interest in this model is again the expected value of discounted future dividend payments until the time of ruin of one of the portfolios,
\begin{align}\label{eq:costfunction}
V(x^{(1)},x^{(2)})=\E_{x^{(1)},x^{(2)}}\left(\int_0^\tau e^{-\delta t}(c^{(1)}+c^{(2)})\1_{E_4}(X^{(1)}_t,X^{(2)}_t)\,dt\right),
\end{align}
for $x^{(1)}\le b^{(1)}$, $x^{(2)}\le b^{(2)}$, with $\tau=\inf\{t\ge 0\colon  (X^{(1)}_t,X^{(2)}_t)\in E_5\}$, 
and $\delta >0$ being a preference-based discount factor.

\section{Iterated integrals and a fixed point approach}\label{sec:fixed-point-approach}
In this section we derive a method for numerical approximation of the quantities of interest appearing in the models introduced in the previous section.
We rewrite the quantity of interest as a sum of integrals with fixed dimension and an error term that goes to zero exponentially fast
with increasing dimension of the integral. This allows for the use of deterministic integration rules. 
The starting point for the derivation of this integral
representation is the observation that the quantity of interest is a fixed point of a certain integral operator associated to the PDMP.

\begin{definition}\label{def:cemetary}
Suppose there exists a set $K^c\subseteq K$ such that 
for all $k\in K^c$ it holds that $\lambda_k(x)=0$,
and $\phi_k(x,t)=x$ for all $x\in E_k$ and all $t\in \R$.
We call $E^c:=\bigcup_{k\in K^c} E_k$ a {\em cemetery} of the PDMP.
\end{definition}

\begin{definition}\label{def:cost-functional}
Let a PDMP be given and let  $E^c\neq\emptyset$ be a cemetery of the PDMP.
A {\em running reward function}
$\ell\colon E\to \R$ is a measurable function satisfying $\ell|_{E^c}\equiv
0$. A {\em terminal cost function} $\Psi\colon E^c\to\R$ is a measurable function satisfying $\Psi|_{E\backslash E^c}\equiv
0$. The {\em cost functional} $V\colon E\to\R$ corresponding to 
$E^c,\ell,\Psi$ is defined by
\begin{align}\label{eq:def_V}
V(x)=\E_x\left(\int_0^\tau e ^{-\delta t} \ell(X_t) dt+e^{-\delta\tau}\Psi(X_\tau)\right),
\end{align}
where $\tau=\inf\{t\geq 0 \colon X_t\in E^c\}$.
\end{definition}

Let $T_1$ be the first jump time. Equation \eqref{eq:def_V} can be rewritten as follows,
\begin{align*}
V(x)=&\E_x\Bigg[
\left(\int_0^{T_1}e^{-\delta t}\ell(\phi(x,t))dt
+\int_{T_1}^{\tau }e^{-\delta t}\ell(\phi(X_{T_1},t-T_1))dt+e^{-\delta\tau }\Psi(X_{\tau })\right)\1_{\{T_1<\tau \}}\\
&\quad+\left(\int_0^{\tau }e^{-\delta t}\ell(\phi(x,t))dt+e^{-\delta\tau }\Psi(\phi(x,\tau ))\right)\1_{\{\tau <T_1\}}\\
&+\left(\int_0^{T_1}e^{-\delta t}\ell(\phi(x,t))dt+e^{-\delta T_1}\Psi(X_{T_1})\right)\1_{\{T_1=\tau \}}
\Bigg].
\end{align*}
Since $X$ is a PDMP and hence a strong Markov process, this yields
$V=\mathcal{H}+\mathcal{G}V$ with $\mathcal{H}\colon E \to\R$, $\mathcal{G}\colon C^2(E,\R)\to \R$ defined by 
\begin{align}\label{eq:induction}
\mathcal{H}(x)&=\E_x\Bigg[
\left(\int_0^{T_1}e^{-\delta t}\ell(\phi(x,t))dt\right)\1_{\{T_1<\tau \}}\nonumber\\
&\quad+\left(\int_0^{\tau }e^{-\delta t}\ell(\phi(x,t))dt+e^{-\delta\tau }\Psi(\phi(x,\tau ))\right)\1_{\{\tau <T_1\}}\nonumber\\
&\quad+\left(\int_0^{T_1}e^{-\delta t}\ell(\phi(x,t))dt+e^{-\delta T_1}\Psi(X_{T_1})\right)\1_{\{T_1=\tau \}}\Bigg],\nonumber\\
\mathcal{G}V(x)&=\E_x\Bigg[e^{-\delta T_1} V(X_{T_1})\1_{\{T_1<\tau \}}\Bigg].
\end{align}
Recall that  for every $t\ge0$ it holds that $\P_x(T_1>t)=\exp\!\big(-\int_0^t\lambda(\phi(x,s))ds\big)=: 1-F_W(t,x)$
and denote the corresponding density by $f_W$.
With this, the function $\mathcal{H}$ and the operator $\mathcal{G}$ admit representations as integrals,
\begin{align*}
\mathcal{H}(x)&=\int_0^{t^\ast (x)}f_W(t,x)\left[\int_0^t e^{-\delta s}\ell(\phi(x,s))ds+e^{-\delta t}\int_{E^c}\Psi(y)Q(dy,\phi(x,t))\right]dt\\
&+(1-F_W(t^\ast (x),x))\left[\int_0^{t^\ast (x)}e^{-\delta s}\ell(\phi(x,s))ds+e^{-\delta t^\ast (x)}\Psi(\phi(x,t^\ast (x)))\right],\\
\mathcal{G}V(x)&=\int_0^{t^\ast (x)}f_W(t,x)e^{-\delta t}\int_E V(y)Q(dy,\phi(x,t))dt.
\end{align*}
Note that $\mathcal{H}(x)$ corresponds to the expected discounted rewards collected before the first jump at time 
$T_1$ when starting in $x$. $\mathcal{G}V(x)$
represents the expected discounted rewards from time $T_1$ onwards conditional on the event 
$\{X_{T_1}\notin E^c, X_0=x\}$. Iterating the above steps $n\in\N$ times leads to
\begin{align}\label{eq:finalV}
V(x)=\mathcal{G}^n V(x)+\sum_{i=0}^{n-1}\mathcal{G}^i\mathcal{H}(x).
\end{align}
\begin{lemma}
\label{lem:Gnto0}
Let $\Psi\colon E^c\to\R$ and $\ell\colon E\to \R$ be bounded, for all $k\in K$ assume that the functions
$\lambda_k$ are bounded by $C_\lambda\in(0,\infty)$,
and for all $x\in E$ let $t^\ast (x)=\infty$. Then for all $x\in E$ and for all $n\in \N$ it holds that $\left|\mathcal{G}^n V(x)\right|\leq C_V\left(C_\lambda/(C_\lambda+\delta)\right)^n$ and, in particular, it holds that $\lim_{n\to\infty} \mathcal{G}^n V(x)=0$ uniformly in $x\in E$.
\end{lemma}
\begin{proof}
The boundedness of $\ell$ and $\Psi$ implies that also $V$ is bounded by
$C_V=\frac{\|\ell\|_\infty}{\delta}+\|\Psi\|_\infty$.
Using the strong Markov property and Equation \eqref{eq:induction} we have by induction on $n$,
\begin{align}\label{eq:Gn}
\mathcal{G}^n V(x)&=
\E_x \left[e^{-\delta T_1} \mathcal{G}^{n-1} V(X_{T_1})1_{\{T_1<\tau\}}\right]\nonumber\\
&= \E_x \left[e^{-\delta T_1}  \E_{X_{T_1}}\left[e^{-\delta (T_n-T_1)} V(X_{T_n})1_{\{T_n <\tau\}}\right] 1_{\{T_1<\tau\}}\right]\nonumber\\
&= \E_x \left[\E_{X_{T_1}}\left[e^{-\delta T_n} V(X_{T_n})1_{\{T_n <\tau\}}1_{\{T_1<\tau\}}\right] \right]\nonumber\\
&=\E_x\left[e^{-\delta T_n}V(X_{T_n}) \1_{\{\tau >T_n\}}\right],
\end{align}
where we used $1_{\{T_n <\tau\}}1_{\{T_1<\tau\}}=1_{\{T_n < \tau\}}$ in the last equality.
Recall that $\P(T_n-T_{n-1}>t\,\vert\,T_{n-1},\,X_{T_{n-1}})=\exp\!\big(-\int_0^t\lambda(\phi(s,X_{T_{n-1}}))ds\big)\geq \exp(-t\,C_\lambda)$.
For every $n\in\N$ let $Z_n\sim\text{Erlang}(n,C_\lambda)$ be an Erlang-distributed random variable.
Combining this with \eqref{eq:Gn} we get that
\begin{align*}
\left|\mathcal{G}^n V(x)\right|\leq C_V\E_x\left[e^{-\delta T_n}\right]\leq C_V\E\left[e^{-\delta Z_n}\right]=C_V\left(\frac{C_\lambda}{C_\lambda+\delta}\right)^n.
\end{align*}
The latter expression converges to zero as $n\to\infty$ uniformly in $x\in E$.
\end{proof}
Combining Lemma \ref{lem:Gnto0} with \eqref{eq:finalV} results in the error estimate
\begin{align}\label{eq:bias}
\left|\,V(x)-\sum_{i=0}^{n-1}\mathcal{G}^i\mathcal{H}(x)\right|\leq C_V\left(\frac{C_\lambda}{C_\lambda+\delta}\right)^n.
\end{align}
Finally, we obtain the following representation,
\begin{equation}\label{eq:iterate}
\begin{aligned}
\mathcal{G}^{i-1}\mathcal{H}(x_0)=&\int\limits_{t_1=0}^{t^\ast (x_0)}f_W(t_1,x_0)e^{-\delta t_1}\hspace{-1.5ex}
\int\limits_{x_1\in E}\int\limits_{t_2=0}^{t^\ast (x_1)}f_W(t_2,x_1)e^{-\delta t_2}
\hspace{-1.5ex}\int\limits_{x_2\in E}\cdots\int\limits_{t_{i-1}=0}^{t^\ast (x_{k-2})}f_W(t_{i-1},x_{i-2})e^{-\delta t_{i-1}}\hspace{-1.5ex}
\\
&\int\limits_{x_{i-1}\in E}\mathcal{H}(x_{i-1}) Q(dx_{i-1},\phi(x_{i-2},t_{i-1}))dt_{i-1}\cdots Q(dx_1,\phi(x_0,t_1))dt_1\\
=&\int\limits_{t_1=0}^{t^\ast (x_0)}\int\limits_{x_1\in E}\cdots\int\limits_{t_{i-1}=0}^{t^\ast (x_{i-2})}\int\limits_{x_{i-1}\in E}
\left(\prod_{j=1}^{i-1}f_W(t_j,x_{j-1})e^{-\delta t_j}\right) \\
&\mathcal{H}(x_{i-1}) Q(dx_{i-1},\phi(x_{i-2},t_{i-1}))dt_{i-1}\cdots Q(dx_1,\phi(x_0,t_1))dt_1.
\end{aligned}
\end{equation}
In \eqref{eq:iterate} we denote by $\{t_j\}_{j\in\{1,\ldots,i-1\}}$ the family of inter-jump times and by $\{x_j\}_{j\in\{1,\ldots,i-1\}}$ the family of post-jump locations. 
\begin{remark}
Solving the integral $\mathcal{G}^{i-1}\mathcal{H}(x_0)$ brings several advantages
compared to a crude Monte Carlo approach. First, \eqref{eq:iterate} is an integral with a fixed dimension. Hence, it can be approximated using deterministic integration rules like quasi-Monte Carlo, for which deterministic error bounds are available.
Second, the bias of restricting oneself to a fixed number of jumps can be estimated uniformly in $x=x_0$ using 
the bias estimate in Lemma \ref{lem:Gnto0}.
Third, rare events like surviving a large number of jumps are---in this formulation---not rare in the sense that
it is unlikely to draw such a realisation, which has the effect of importance sampling.
\end{remark}

\section{Cubature rules for $C^\kappa$-functions}\label{sec:cubature}
In order to obtain convergence estimates for numerical integration methods such as quasi-Monte Carlo (QMC) methods or other 
cubature rules, we need more regularity of the integrands than they admit in many practical applications.
For example, we may need to bound a certain norm of the Hessian matrix of the integrand. In
Section \ref{sec:smoothing}, we will rewrite the problem introduced in Section \ref{sec:fixed-point-approach} so that the integrand is a function $f\colon [0,1]^d\rightarrow \R$ 
which satisfies $f\in C^2([0,1]^d)$, or more generally $f\in C^\kappa ([0,1]^d)$ for some $\kappa \in\N$.
We outline two different methods for treating such integrands $f$ by cubature rules.

\subsection{Quasi-Monte Carlo methods}
Quasi-Monte Carlo methods are equal-weight cubature rules with $M$ deterministically chosen integration nodes.
Let the integrand $f\colon [0,1]^d\rightarrow \R$ satisfy $f\in C^2([0,1]^d)$. 
In order to obtain a convergence estimate for numerical integration of $f$ using QMC, we require a so-called Koksma-Hlawka type inequality.  
The original Koksma-Hlawka inequality bounds the integration error of a QMC rule by the product of the variation of the integrand (in the sense of Hardy and Krause) and the so-called discrepancy of the integration node set (see, e.g., \cite[Chapter 2]{niederreiter1992}). 
We remark, however, that we cannot easily apply the classical Koksma-Hlawka inequality in this paper,
as we cannot rely on the integrands to have bounded variation in the sense of Hardy and Krause.
Hence, we are going to resort to a variant of the Koksma-Hlawka inequality which was recently proven in \citet{pausingersvane2015}.  
Let $Q_{M,d} (f)=\frac{1}{M}\sum_{j=1}^M f(\bsx_j)$ be a QMC rule using $M$ integration nodes $\bsx_1,\ldots,\bsx_M \in [0,1)^d$.
Then by \cite[Theorem 3.12]{pausingersvane2015} we have 
\begin{equation}\label{eqpausva}
 \abs{\int_{[0,1]^d} f(\bsx) d \bsx - Q_{M,d} (f)}\le 
 \left(\sup_{\bsx\in[0,1]^d} f(\bsx) - \inf_{\bsx\in [0,1]^d} f(\bsx) + \frac{d}{16} M(f)\right) \mathrm{Disc}_{\mathrm{I}} (\bsx_1,\ldots,\bsx_M),
\end{equation}
where $M(f)=\sup_{\bsx\in [0,1]^d}\norm{\mathrm{Hess}(f,\bsx)}$, $\mathrm{Hess} (f,\bsx)$ is the Hessian matrix of $f$ at $\bsx$, $\norm{\cdot}$ denotes 
the usual operator norm, and where $\mathrm{Disc}_{\mathrm{I}}(\bsx_1,\ldots,\bsx_M)$ is the isotropic discrepancy of the integration node set,
\[
 \mathrm{Disc}_{\mathrm{I}}(\bsx_1,\ldots,\bsx_M)=\sup_{\substack{ C\subseteq [0,1]^d\\  C\ \mathrm{convex}}}
 \abs{\frac{1}{M}\sum_{j=1}^M \1_{\{\bsx_j\in  C\}} - \mu_d ( C)},
\]
where $\mu_d$ denotes the Lebesgue measure on the $\R^d$. Now let $\bsx_1,\ldots,\bsx_M\in [0,1]^d$. In \cite[Chapter 2]{niederreiter1992} it is shown that
\[\mathrm{Disc}_{\mathrm{I}} (\bsx_1,\ldots,\bsx_M) \le 8 d \left(\mathrm{Disc}_{\ast} (\bsx_1,\ldots,\bsx_M)\right)^{1/d},\]
where by $\mathrm{Disc}_{\ast} (\bsx_1,\ldots,\bsx_M)$ we denote the star discrepancy of $\bsx_1,\ldots,\bsx_M$, defined as 
\[
 \mathrm{Disc}_{\ast}(\bsx_1,\ldots,\bsx_M) =\sup_{\boldsymbol{a}\in [0,1)^d}
 \abs{\frac{1}{M}\sum_{j=1}^M \1_{\{\bsx_j\in [\boldsymbol{0},\boldsymbol{a})\}} - \mu_d ([\boldsymbol{0},\boldsymbol{a}))},
\]
where $[\boldsymbol{0},\boldsymbol{a})$ denotes $[0,a_1)\times \cdots \times [0,a_d)$ for $\boldsymbol{a}=(a_1,\ldots,a_d)$.
It is well known that common point sequences that are employed in QMC methods, such as Sobol' sequences or Halton sequences, have 
a star discrepancy of order $(\log M)^d /M$ (and it is known that this order can, if at all, 
only be improved with respect to the exponent of the $\log$-term). 
Hence, by using, e.g., Sobol' points in a QMC method for numerically integrating a $C^2$-function, 
we cannot expect an error that converges to zero faster than 
$(\log M)/M^{1/d}$.\\
As we shall see below, this order of magnitude can, with respect to the disadvantageous dependence on $d$, not be
improved further for $C^2$-functions. However, there is room for improvement if we make additional smoothness assumptions on 
the integrand. 

\subsection{Product rules}

In \citet{Hinrichs2017} it is shown that, by using products of Gauss rules, one can obtain the following result. 
Let $f\colon [0,1]^d\rightarrow\R$ be such that $f\in C^\kappa $ for some $\kappa \in\N$. Then, by using a product rule $Q_{G,\tilde M,d}$ of $\tilde M$-point Gauss quadrature rules, one obtains 
\begin{equation}\label{eqhnuw}
 \abs{\int_{[0,1]^d} f(\bsx) d \bsx - Q_{G,\tilde M,d} (f)}\le c_\kappa  d \tilde M^{-\kappa } \norm{f}_{C^\kappa },\quad \mbox{for}\quad \tilde M\ge \kappa +1,
\end{equation}
where $c_\kappa =(\pi/2) (\mathrm{e}/(6\sqrt{3}))^\kappa $, and where
\[
 \norm{f}_{C^\kappa }=\max_{\substack{\boldsymbol{\beta}\in \N_0^d\\ \norm{\boldsymbol{\beta}}_{1}\le \kappa }} 
\norm{D^{\boldsymbol{\beta}} (f)}_{L_\infty},
\]
where $D^{\boldsymbol{\beta}}$ denotes the (weak) partial derivative of order $\boldsymbol{\beta}$ for $\boldsymbol{\beta}\in \N_0^d$.
A $d$-fold Gauss product rule as described above uses $M=\tilde M^d$ points in total, and hence yields
a convergence order of $M^{-\kappa /d}$. It is known due to \cite{bakhvalov1959} that this convergence order is best possible. 
For the special case $\kappa =2$, we only obtain a relatively small improvement over the bound implied by \eqref{eqpausva}.
However, there is an additional advantage in the  bound \eqref{eqhnuw}. By requiring that the function $f$ 
satisfies additional smoothness assumptions, namely 
that $f\in C^\kappa $ for some $\kappa \in\N$ which is possibly larger than 2, we obtain an improved convergence rate.
Hence, we face a trade-off between imposing a higher degree of smoothness on the integrand $f$ to obtain a higher accuracy 
in the quadrature rule, and the error we make by smoothing the integrand to that extent. 
It is therefore likely that the method needs to be fine-tuned on a case-by-case basis.
In practice, product rules often cannot be applied, since, for example, for integrating a $d=1024$-variate integrand using only $\tilde M=2$ 
integration nodes per coordinate requires a point set consisting of $M=2^{1024}$ integration nodes. To overcome the latter problem, it might 
be useful to apply the theory of weighted integration as introduced in \cite{SW98}, possibly combined with truncation (see, e.g., \cite{KPW16}) or 
multivariate decomposition methods (see, e.g., \cite{KSWW10}). A detailed analysis of these approaches applied to the present problem is left open for 
future research.

\section{Smoothing of the integrand}\label{sec:smoothing}
The integrand in \eqref{eq:iterate} is not necessarily a $C^\kappa$-function.
Therefore, in this section we provide a technique 
for smoothing the integrand in order to apply
convergence results for integration rules that are described in Section
\ref{sec:cubature}.

The piecewise construction of the process described in Definition \ref{def:PDMP} leads to the situation that $X_{t}=\phi(X_{T_{j-1}},t-T_{j-1})$
for $t\in[T_{j-1},T_j)$ is a function of $X_{T_{j-1}}$ and $T_{j-1}$. In particular, all subsequent pre-jump locations depend on all
previous post-jump locations and jump times, via $\phi$ and $\lambda$. Consequently, regularity of the integrand depends on regularity of the flow $\phi$ and the intensity function $\lambda$.
The analysis in this section is restricted to the case where the flow originates
from autonomous ODEs, i.e.,~for all $k\in K$ there exist Lipschitz 
continuous functions $\odecoeff_k:\R^{d(k)}\to \R^{d(k)}$ such that 
$\frac{\partial}{\partial t}\phi_k(y,t)= \odecoeff_k(\phi_k(y,t))$.
General results from the literature on ODEs, see, e.g.,~\cite{Grigorian}, yield that the derivatives
$\frac{\partial}{\partial y}\phi_k,\,\frac{\partial^2}{\partial y^2}\phi_k,\,\frac{\partial}{\partial t}\phi_k$
can be described by so-called associated first and second order variational equations for
which one requires  $\odecoeff_k$ to be a $C^2$-function.

For the density $f_W$ of the inter-jump times to be $C^2$ we need that
$\lambda\in C^2(E,\R)$. Also we need $\ell\in C_b^2(E,\R)$, and $\Psi\in C_b^2(E,\R)$ since they appear in the integral defining $\cH$. 

A serious problem with respect to smoothness arises if the PDMP model allows for jumps from the active 
boundary. Suppose $(k,y)\in E$ and $t^*(k,y)<\infty$.
Then, conditional on $X_t=(k,y)$, the time of the next jump is distributed as
$\min(T,t^*(k,y))$, where $T$ has distribution function $F_T(t)=
1-\exp(-\int_0^t\lambda_k(\phi_k(y,s))ds)$. 
But in general neither $t^*(k,y)$ nor
 $\min(T,t^*(k,y))$ will depend
smoothly on $y$, even if $\lambda_k$ has arbitrarily high regularity. 
We are not aware of a general remedy for this problem.  
However, for all PDMP models put forward in Section \ref{sec:examples}, 
the jumps from the active boundary do not constitute jumps of the original problem. 
In the following subsection we describe by example how PDMPs can be approximated by PDMPs that do not allow for jumps from the boundary.

Concerning the jump kernel $Q$, it is hard to state general sufficient regularity conditions.
An exemplary favourable situation arises if the jump kernel $Q$ admits a $C^2$-density $f_Y$ in the sense that $Q(A,x)=\int_A f_Y(x_1,x)dx_1$
for all $A\in\mathcal{E}$ and all $x\in E$. In the one-dimensional examples from risk theory in Sections \ref{subsubsec:C-L}--\ref{subsubsec:CL-loan}, this condition is satisfied and for
the two-dimensional example in Section \ref{subsubsec:C-L-multidim} we present a smoothing technique in Section \ref{sec:smoothQ}.

\subsection{Smoothing of the flow}

Consider the example from Section \ref{subsubsec:CL-loan} without dividend 
barrier. We can describe the problem alternatively with a state space
consisting of three components: 
\begin{itemize}
\setlength{\itemsep}{0em}
\item $K=\{1,2,3\}$, 
\item $E_1=(-\frac{c}{\rho},\infty)$, $E_2=(-\infty,-\frac{c}{\rho})$, $E_3=\{-\frac{c}{\rho}\}$, 
\item 
$\phi_1$ is determined by an autonomous ODE of the form 
$\odecoeff_1:\R\to \R$,
\begin{align}\label{eq:ode-coeff-de}
\odecoeff_1(y)&=\left\{
\begin{array}{ll}
c,& \mbox{if}\ y\in(0,\infty),\\
c+\rho y ,& \mbox{if}\ y\in(-\frac{c}{\rho},0],\\
0,& \mbox{if}\ y\in(-\infty,-\frac{c}{\rho}],
\end{array}
\right.
\end{align}
for some $c>0,\,\rho>0$. The function $\phi_2$ is given by
$\phi_2(y,t)=y$ $\forall y\in E_2$ and $\forall t\in \R$,
and $\phi_3$ by 
$\phi_3(y,t)=y$ $\forall y\in E_3$ and  $\forall t\in \R$, 
\item $\lambda_1(y)=\lambda_N$ $\forall y\in E_1$, 
$\lambda_2(y)=0$ $\forall y\in E_2$, 
$\lambda_3(y)=0$ $\forall y\in E_3$. 
\item For 
$B=(\{1\}\times B_1)\cup(\{2\}\times B_2)\cup(\{3\}\times B_3)\in \mathcal{E}$, 
\begin{align*}
Q(B,(1,y))& =\P(Y\in y-B_1)+\P(Y\in y-B_2)+\P(Y\in y-B_3)\,  &(\mbox{for}\ y\in E_1)\,,\\
Q(B,(2,y))&=\P(Y\in y-B_2)\,& (\mbox{for}\ y\in E_2)\,,\\
Q(B,(3,y))&=\P(Y\in y-B_2)+\P(Y\in y-B_3)\,&(\mbox{for}\ y\in E_3)\,.
\end{align*}
\end{itemize}
Here, $\odecoeff_1$ is not differentiable in $0$. However,
we may smoothen this discontinuity using a `smoothened Heaviside function'.
Note that $\Gamma^*=\emptyset$.
\begin{definition}
Let $\kappa \in \N\cup\{0\}$.
We call a function 
$\sheavi\colon \R\to [0,1]$ a $C^\kappa$-Heaviside function,  
if 
\begin{itemize}
\setlength{\itemsep}{0em}
\item $\sheavi(y)=0$ for $y<-1$,
\item $\sheavi(y)=1$ for $y>1$,
\item $\sheavi$ is non-decreasing,
\item $\sheavi(y)+\sheavi(-y)=1$,
\item $\sheavi$ is $\kappa$-times continuously differentiable.
\end{itemize}
\end{definition}

\begin{lemma}
Let $\kappa \in \N\cup\{0\}$, and let $f\colon \R\to\R$ be a piecewise $C^\kappa$-function with discontinuity in $\xi\in \R$, i.e.,~let 
there exist $C^\kappa$-functions $f_1,f_2\colon \R\to \R$ with
$f=f_1$ on $(-\infty,\xi)$ and $f=f_2$ on $(\xi,\infty)$.
Let $\sheavi$ be a $C^\kappa$-Heaviside function.
For every $\varepsilon>0$ define  $f^\varepsilon\colon \R\to\R$ by
$f^\varepsilon(y)=f_1(y)\sheavi(\frac{y-\xi}{\varepsilon})+f_2(y)\sheavi(-\frac{y-\xi}{\varepsilon})$.
Then,
\begin{enumerate}[(i)]
\item $f^\varepsilon\in C^\kappa$ 
for every $\varepsilon>0$,
\item  $f^\varepsilon\big|_{\R\backslash (-\varepsilon,\varepsilon)}=
f\big|_{\R\backslash (-\varepsilon,\varepsilon)}$
for every $\varepsilon>0$,
\item for all $y\in \R\backslash \{\xi\}$ it holds that  $\lim_{\varepsilon\to 0+} f^\varepsilon(y)=f(y)$, 
\item for all $\delta>0$ it holds that 
 $\lim_{\varepsilon\to 0+}\sup_{y\in \R\backslash(\xi-\delta,\xi+\delta)}|f^\varepsilon(y)-f(y)|=0$.
\end{enumerate}
\end{lemma}
\begin{proof}
The elementary proof is left to the reader.
\end{proof}
There are various possible choices for the smoothing:~from the left 
$f^{\varepsilon-}(y)=f_1(y)\sheavi(\frac{y-\xi+\varepsilon}{\varepsilon})+f_2(y)\sheavi(-\frac{y-\xi+\varepsilon}{\varepsilon})$ 
and from the right
$f^{\varepsilon+}(y)=f_1(y)\sheavi(\frac{y-\xi-\varepsilon}{\varepsilon})+f_2(y)\sheavi(-\frac{y-\xi-\varepsilon}{\varepsilon})$.
Figure \ref{fig:smooth} depicts these three possible smoothings for a function with a discontinuity in $\xi=1$.
\begin{figure}
\begin{center}
\includegraphics{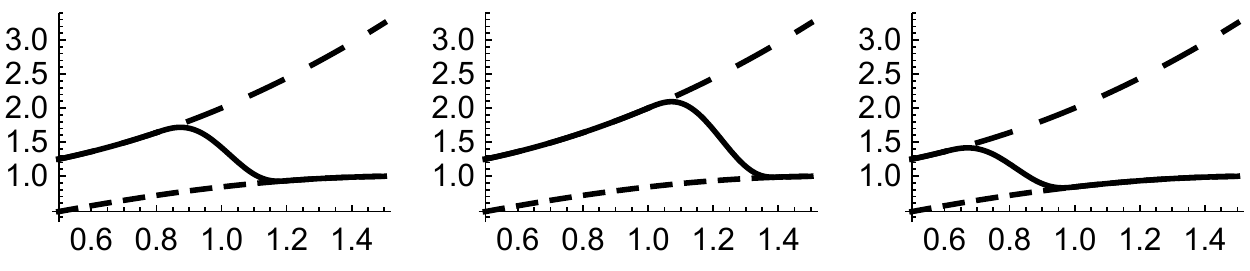}
\end{center}
\caption{Illustration of smoothing a piecewise $C^2$-function with a discontinuity in $\xi=1$.}\label{fig:smooth}
\end{figure}
A concrete example for a function $\sheavi$ that satisfies the above requirements is given by
\begin{equation}\label{eq:sheavi}
\sheavi(y)=
\begin{cases}
0 & \mbox{if}\ y<-1,\\
\frac{1}{2} + \frac{15 y}{16 } - \frac{5 y^3}{8 } + 
\frac{3 y^5}{16 }  & \mbox{if}\ y\in [-1,1],\\
1 & \mbox{if}\ y>1.
\end{cases}
\end{equation}
For every $\veps>0$, a  smoothed version of the function $\odecoeff_1$ defined
in \eqref{eq:ode-coeff-de} is given by
\begin{align*}
\odecoeff_1^\varepsilon(y)&=(c+\rho\, y)\,\sheavi\Big(-\frac{y}{\varepsilon}\Big)+c\,\sheavi\Big(\frac{y}{\varepsilon}\Big)\,.
\end{align*}
We can finally formulate a PDMP corresponding to the new model, where the flow has been smoothened,
\begin{itemize}
\setlength{\itemsep}{0em}
\item $K=\{1,2,3\}$, 
\item $E_1=(-\frac{c}{\rho},\infty)$, $E_2=(-\infty,-\frac{c}{\rho})$, $E_3=\{-\frac{c}{\rho}\}$,
\item $\frac{\partial}{\partial t}\phi_1^\varepsilon(y,t)=\odecoeff_1^\varepsilon(\phi_1^\varepsilon(y,t))$ 
$\forall y\in E_1$ and $\forall t\in \R$, 
$\phi_k(y,t)=y$ $\forall y\in E_k$ and $\forall t\in \R$,  $k\in\{2,3\}$;
\item $\lambda_1(y)=\lambda_N$  $\forall y\in E_1$, 
$\lambda_k(y)=0$  $\forall y\in E_k$, $k\in\{2,3\}$;
\item for 
$B=(\{1\}\times B_1)\cup(\{2\}\times B_2)\cup(\{3\}\times B_3)\in \mathcal{E}$, 
\begin{align*}
Q(B,(1,y))& =\P(Y\in y-B_1)+\P(Y\in y-B_2)+\P(Y\in y-B_3)\,  &(\mbox{for}\ y\in E_1)\,,\\
Q(B,(2,y))&=\P(Y\in y-B_2)\,& (\mbox{for}\ y\in E_2)\,,\\
Q(B,(3,y))&=\P(Y\in y-B_2)+\P(Y\in y-B_3)\,&(\mbox{for}\ y\in E_3)\,.
\end{align*} 
\end{itemize}
Note that $\Gamma^*=\emptyset$.
Since the dividend barrier $b$ is never reached, we also have to smoothen the reward function in a way that the region where dividends are paid 
can be reached, i.e.,
$\ell^\veps(y)=c\, \sheavi(\frac{y-b+\varepsilon}{\varepsilon})$. We will show 
convergence of the corresponding value functions in 
Section \ref{sec:DE}.\\

\subsection{Smoothing of jump measures}\label{sec:smoothQ}

We give an example for a class of jump measures that can be
approximated by measures leading to $C^2$-integrands in \eqref{eq:iterate}.

Let $(E,\mathcal{E})$ be the state space of a PDMP and let 
$(\phi,\lambda,Q)$ be its local characteristics. 
Let the probability kernel $Q$ satisfy the following assumption.
\begin{assumption}\label{ass:jumpkernel}
We assume that
\begin{enumerate}
\item\label{ass:jumpkernel-ass1} there exists a positive integer $n$ such that 
for every $k\in K$,  and every $y\in E_k$,
 there exist sets 
$B_{1}(k,y),\ldots,B_{n}(k,y)$ such that  
\begin{enumerate}[(i)]
\item for every $j\in \{1,\ldots,n\}$ there exists $k_1\in K$ such that
$B_{j}(k,y)\subseteq E_{k_1}$,
\item for every $j\in
\{1,\ldots,n\}$ it holds that $\{(\bar y,z)\colon\bar y\in E_k, z\in
B_{j}((k,\bar y))\}$ is a connected $C^2$-manifold,
\end{enumerate}
\item for every $k\in K$ and every $j\in \{1,\ldots,n\}$ 
the mapping from $E_k$ to $\R$, $\bar y \mapsto Q(B_{j}((k,\bar y),x)$ is $C^2$,
\item for all $x\in E$ it holds that $\sum_{j=1}^n Q(B_{j}(x),x)=1$,
\item for every $x\in E$ and every $j\in \{1,\ldots,n\}$ there exists
a $C^2$-mapping $G_{j,x}\colon [0,1]^{\dim(B_{j})}\to  B_{j}$ such that 
for all $A\in \mathcal{E}$ it holds that
\[
Q(A\cap B_{j},x)=\mu_{\dim(B_{j})}(G_{j,x}^{-1}(A\cap B_{j}))Q(B_{j},x),
\]
where $\mu_m$ denotes the $m$-dimensional Lebesgue measure,
\item for every $k\in K$ and every $j\in \{1,\ldots,n\}$ the mapping from
$E_k\times [0,1]^{\dim(B_{j})}$ to $\bigcup_{l\in K} E_l$,
$(y,u)\mapsto G_{j,(k,y)}(u)$ is $C^2$.
\end{enumerate}
\end{assumption}
Note that Assumption \ref{ass:jumpkernel}.\ref{ass:jumpkernel-ass1} implies that, for every $x\in E$,  $B_{j}(x)$
is a $C^2$-manifold, and that 
for all $x_1=(k_1,y_1),x_2=(k_2,y_2)\in E$ with $k_1=k_2$ we have $\dim  B_{j}(x_1)=\dim B_{j}(x_2)$.

Under Assumption \ref{ass:jumpkernel} we have for $x\in E$ and for $f\in C^2_b(E,\R)$ that
\[
\int_E f(y) Q(dy,x)= \sum_{j=1}^n p_{j}(x)\int_{[0,1]^{\dim (B_{j}(x))}} f(G_{j,x}(u))du,
\]
where $p_{j}(x)=Q(B_{k,j},x)$ for all $x\in E$.
For the integral in \eqref{eq:iterate}
this implies that we have iterated sums for each 
jump, which increases the complexity for large numbers of jumps. 
Instead, we may write the sum as an integral
over $[0,1]$,
\[
\int_E f(y) Q(dy,x)= \int_0^1 \sum_{j=1}^n 
\1_{[q_{k,j-1}(x),q_{k,j}(x))}(u_0)\int_{[0,1]^{\dim (B_{j}(x))}} 
f(G_{j,x}(u))du\,du_0,
\]
where $q_{0}(x)=0$ and $q_{j}(x)=p_{1}(x)+\dots+p_{j}(x)$.
However, with this `trick' we have lost 
the property of the integrand being $C^2$.
So, using again our smoothened Heaviside function 
$\sheavi\colon \R\to [0,1]$, we can smoothen the indicator functions,
\begin{align*}
\lefteqn{\int_E f(y) \Qsmooth(dy,x)}\\
&= \int_0^1 \sum_{j=1}^n \Big(\sheavi\Big(\frac{u_0-q_{j-1}(x)}{\varepsilon}\Big)+\sheavi\Big(\frac{q_{j}(x)-u_0}{\varepsilon}\Big)\Big)\int_{[0,1]^{\dim (B_{j}(x))}} f(G_{j,x}(u))du\,du_0\\
&= \int_0^1 \int_{[0,1]^{{\dim(B_{j}(x))}}}\sum_{j=1}^n \Big(\sheavi\Big(\frac{u_0-q_{j-1}(x)}{\varepsilon}\Big)+\sheavi\Big(\frac{q_{j}(x)-u_0}{\varepsilon}\Big)\Big) f(G_{j,x}(u_1,\ldots,u_{\dim(B_{j}(x))}))du\,du_0\,.
\end{align*}
This expression, considered as a function of $x$, is $C^2$ as it is a composition 
of $C^2$-functions.
\begin{theorem}
In the setup of this section we have for all $f\in C^0_b(E,\R)$ that 
\begin{align*}
\lim_{\varepsilon\to 0}\int_E f(y) \Qsmooth(dy,x)=\int_E f(y) Q(dy,x).
\end{align*}
\end{theorem}
\begin{proof}
It holds that
\begin{align*}
\lefteqn{\Big|\int_E f(y) (\Qsmooth(dy,x)-Q(dy,x))\Big|}\\
&=\Big|\sum_{j=1}^n \int_0^1 \Big(\sheavi\Big(\frac{u_0-q_{j-1}(x)}{\varepsilon}\Big)+\sheavi\Big(\frac{q_{j}(x)-u_0}{\varepsilon}\Big)-1_{[q_{j-1}(x),q_j(x))}(u_0)\Big)du_0\int_{[0,1]^{\dim(B_{j}(x))}} f(G_{j,x}(u))du\Big|\\
&\le \sum_{j=1}^n \int_0^1 \Big|\sheavi\Big(\frac{u_0-q_{j-1}(x)}{\varepsilon}\Big)+\sheavi\Big(\frac{q_{j}(x)-u_0}{\varepsilon}\Big)-1_{[q_{j-1}(x),q_j(x))}(u_0)\Big|du_0\int_{[0,1]^{\dim (B_{j}(x))}} \Big|f(G_{j,x}(u))\Big|du\,.
\end{align*}
For our concrete example of $\sheavi$ the first integral can be estimated 
by $\frac{5}{8}\varepsilon$.
Thus 
\begin{align*}
\Big|\int_E f(y) (\Qsmooth(dy,x)-Q(dy,x))\Big|
\le\frac{5\varepsilon n}{8}\|f\|_\infty\,,
\end{align*}
yielding the statement of the theorem. 
\end{proof}
Now, consider the example from Section \ref{subsubsec:C-L-multidim}. Here, a jump can
be either a jump in $x_1$-direction or a jump in $x_2$-direction, i.e.,
\begin{align*}
X_{T_j}=
\begin{cases}
X_{T_J-}+(Y_1,0) & \text{with probability } \frac{\lambda_1}{\lambda_1+\lambda_2},\\
X_{T_J-}+(0,Y_2) & \text{with probability } \frac{\lambda_2}{\lambda_1+\lambda_2}.
\end{cases}
\end{align*}
In this case we can find functions $G_1,G_2:[0,1]\to [0,\infty)$ such that 
  $Y_1 \stackrel{d}{\sim} G_1(\varTheta_1)$ and 
$Y_2\stackrel{d}{\sim} G_2(\varTheta_2)$ for uniform random variables $\varTheta_1,\varTheta_2$.
Hence,
\begin{align*}
\int_E f(y) Q(dy,(x_1,x_2))\approx 
\int_0^1 \int_{[0,1]^2} &
\sheavi\Big(\varepsilon^{-1}\Big(\frac{\lambda_1}{\lambda_1+\lambda_2}-u\Big)\Big)f\big(x_1+G_1(\vartheta_1),x_2\big)
\\&+
\sheavi\Big(\varepsilon^{-1}\Big(u-\frac{\lambda_1}{\lambda_1+\lambda_2}\big)\Big)f\Big(x_1,x_2+G_2(\vartheta_2)\big)d\vartheta_1\, d\vartheta_2\,du\,.
\end{align*}
\begin{remark}
If we consider, say, $i=100$ in \eqref{eq:iterate}, then we get a very high number of terms to be summed in the integral.
However, we always assume $\varepsilon$ to be very small, in particular,
we may assume that per jump at most two, and in most situations only one, of the terms 
$\sheavi(\varepsilon^{-1}(u-q_{j-1}(x)))+\sheavi(\varepsilon^{-1}(q_{j}(x)-u))$ are  nonzero.
\end{remark}

\subsection{Convergence}\label{sec:FellerCase}
In this section we prove a general convergence result for approximated versions
of PDMPs with smoothing as above. We will exploit results on Feller processes presented in
\citet[Chapter 19]{Kallenberg2002} and \citet[Chapters 4.2 and
4.8]{EthierKurtz1986}.
For the remainder of this section we make the following assumptions:
\begin{enumerate}[(i)]
\item $t^*(x)=\infty$ for all $x\in E$,
\item $\lambda\in C_b(E,\R)$, 
\item \label{it:ass-jump} for all $f\in C_b(E)$ the mapping $x\mapsto\int_E f(\bar x)Q(d\bar x,x)$ is
continuous.
\end{enumerate}

With this, we can utilise the following theorem.
\begin{theorem} \cite[Theorem 27.6]{davis1993} \label{th:feller}
If $t^*(x)=\infty$ for all $x\in E$ and for all $\lambda\in C_b(E,\R)$, and if
the mapping $x\mapsto\int_E f(y)Q(dy,x)$ is continuous for all $f\in C_b(E,\R)$,
then the PDMP is a Feller process.
\end{theorem}
We give an example for a class of jump kernel which comprises the jump kernels of the one-dimensional 
examples in Section \ref{sec:examples} and which satisfies (\ref{it:ass-jump}).

\begin{example}
Let $E_k\subseteq \R$ be an interval for every $k\in K$ and let $f_Y$ be a 
bounded density
function on $\R$. Furthermore, let, for every $x=(k,y)\in E$ and every $A\in\cE$,  
$Q(A,(k,y))=\sum_{j\in K}\int_{(y-A)\cap E_j}f_Y(\bar y)d\bar y$.  
Then for every $f\in C_b(E,\R)$ it holds that 
\begin{align*}
\lefteqn{\left\vert\int_E f(x)Q(dx,(k,y_1))-\int_E f(x)Q(dx,(k,y_2))\right\vert}\\
&=\left\vert\sum_{j\in K}\int_\R \1_{E_j}(y_1-\bar y)f_j(y_1-\bar y)f_Y(\bar y)d\bar y
-\sum_{j\in K}\int_\R \1_{E_j}(y_2-\bar y)f_j(y_2-\bar y)f_Y(\bar y)d\bar y\right\vert\\
&\le\sum_{j\in K}\left\vert\int_\R \1_{E_j}(y_1-\bar y)f_j(y_1-\bar y)f_Y(\bar y)d\bar y
-\int_\R \1_{E_j}(y_2-\bar y)f_j(y_2-\bar y)f_Y(\bar y)d\bar y\right\vert\\
&\le\sum_{j\in K}\int_\R\vert \1_{E_j}(y_1-\bar y)f_j(y_1-\bar y)
-\1_{E_j}(y_2-\bar y)f_j(y_2-\bar y)\vert f_Y(\bar y)d\bar y\,.
\end{align*}
Since, by assumption, all $f_j$ are continuous and all $E_j$ are intervals,
it holds that $\vert \1_{E_j}(y_1-\bar y)f_j(y_1-\bar y)
-\1_{E_j}(y_2-\bar y)f_j(y_2-\bar y)\vert$ is bounded by $2\|f_j\|_\infty$ and
goes to zero as $y_1\to y_2$ for 
almost all $\bar y$.

Therefore, bounded convergence implies that the above sum converges to 0.
From this the desired continuity follows.
\end{example}

The generator of $X$ in the setup of the current section is given by
\begin{align}\label{eq:generator}
\mathcal{A}f(x)=\calX f(x)+\lambda(x)\int_E (f(\bar x)-f(x))Q(d\bar x,x)\,,\quad x\in E,
\end{align}
where for $x=(k,y)\in E$ we define $\calX f(x)$ by  $(\calX f)_k(y)=\frac{\partial}{\partial t}f_k(\phi_k(y,t))|_{t=0}$.
Note that for $f\in  C^1_b(E,\R)$ this means $(\calX f) (y)=g (y)\cdot\nabla f (y)$.
So the domain $\cD(\cA)$ of the generator consists of all functions in 
$C_b(E,\R)$ which are continuously differentiable along the trajectories of the flow on all components, cf.~\citet[page 8]{EthierKurtz1986}, and 
$C^1_b(E,\R)\subseteq \cD(\cA)$.
\begin{definition}[\mbox{\cite[Chapter 19]{Kallenberg2002}}]
Let $A$ be a closed linear operator with domain of definition $\cD(A)$. A {\em core} for $A$ is a linear subspace $D\subseteq \cD(A)$ such that the restriction $A|D$ has closure $A$.
\end{definition}
\begin{proposition}[\mbox{\cite[Proposition 19.9]{Kallenberg2002}}]
If $\cA$ is the generator of a Feller semigroup $(P_t)_{t\ge 0}$, 
then any dense,
$(P_t)_{t\ge 0}$-invariant subspace $D\subseteq \cD(\cA)$ is a core for $\cA$.
\end{proposition}
\begin{proposition}\label{prop:core} Under the assumptions made in this section, and for $\cA$ being defined as in 
\eqref{eq:generator}, it is true that $C_b^\infty(E,\R)$ is a core for $\cA$.
\end{proposition}
\begin{proof}
We certainly have that $C_b^\infty(E,\R)$ is a dense subspace of $C_b(E,\R)$.
Furthermore, the transition semigroup satisfies 
$P_t\colon  C_b(E,\R)\to C_b(E,\R)$
for all $t\in [0,\infty)$, see \cite[p.76]{davis1993}, since the PDMP is Feller by 
Theorem
\ref{th:feller}.

We have to prove that $C^\infty_b(E,\R)$ is invariant under $(P_t)_{t\in
[0,\infty)}$. We show this by proving that, for all $k\in \N$, $P_t  C_b^k(E,\R)\subseteq  C_b^k(E,\R)$.
For $k=0$ this is just the Feller property.
Since all derivatives are bounded in the $\sup$-norm, differentiation and
application of $P_t$ commute, i.e., $\frac{\partial^k}{\partial
x^k}P_tf=P_t\frac{\partial^k}{\partial x^k}f\in C_b(E,\R)$ for all $k\in\N$.
Consequently, $ C_b^\infty(E,\R)$ is a core for $\mathcal{A}$.
\end{proof}
\begin{theorem}[{\cite[Theorem 19.25]{Kallenberg2002}}]\label{th:kallenberg}
Let $X$ be a Feller process in $E$ with semigroup $(P_t)_{t\ge 0}$ and generator 
$\mathcal{A}$ with domain $\mathcal{D}(\cA)$, and for all $n\in\N$ let 
$X^n$ be Feller processes in $E$ with semigroups $(P^n_{t})_{t\ge 0}$ and generators $\mathcal{A}^n$ with domains $\mathcal{D}(\cA^n)$. 
Let $D$ be a core for $\mathcal{A}$. Then the following statements are equivalent: 
\begin{enumerate}[(i)]
\item\label{it:kallenberg1} for every $f\in D$ there exists a sequence $(f^n)_{n\in \N}$ with
 $f^n\in\mathcal{D}(\cA^n)$ for all $n\in\N$ and such that  
$f^n\to f$ and $\mathcal{A}^n f^n\to\mathcal{A}f$ uniformly as $n\to\infty$,
\item for all  $t>0 $ we have $P^n_t\to P_t$ as $n\to \infty$ in the strong operator topology,
\item for every $f\in C_0(E,\R)$ and every $t_0\in(0,\infty)$ it holds that $P^n_tf\to P_t f$  as $n\to \infty$
uniformly for $t\in[0,t_0]$,
\item \label{it:kallenberg4} if $X_0^n\stackrel{d}{\to} X_0$ in $E$, then $X^n\stackrel{d}{\to} X$ in $D([0,\infty),E)$.
\end{enumerate}
\end{theorem}

\begin{remark}\label{rem:skorokhod}
The notion of weak convergence of processes 
in Item \eqref{it:kallenberg4} needs an explanation. Here,
$D([0,\infty),E)$ is the space of c\`adl\`ag functions, equipped with the
Skorokhod topology, see \citet[p.~118]{EthierKurtz1986}.
With this topology, $D([0,\infty),E)$ is a Borel subset of a Polish space and 
for a sequence $(X^n)_{n\in \N}$ of $D([0,\infty),E)$-valued random variables 
(i.e., processes in $E$ with c\`adl\`ag paths), and a 
$D([0,\infty),E)$-valued random variable $X$ we have 
$X^n\stackrel{d}{\to} X$ if and only if $\lim_{n\to \infty} \E(F(X^n))=\E(F(X))$ for 
all bounded Skorokhod continuous functions $D([0,\infty),E)\to \R$,
see
\citet[Section 6]{KurtzProtter1996} or 
\citet[Chapter 3]{EthierKurtz1986}.
We do not wish to go into the details of the notion of 
Skorokhod continuity. It suffices to mention that from \cite[Section 8, Example  8.1]{KurtzProtter1996}
we know that
for given continuous functions 
$f_1:E\times [0,\infty)\to\R^d$ and $f_2: [0,\infty)\to [0,\infty)$, and fixed 
$t\in [0,\infty)$, the following functionals exhibit this property:
\begin{align*}
F_1(\omega)&=f_1(\omega(t),t)&&(\mbox{for}\ \omega\in D([0,\infty),E)),\\
F_2(\omega)&=\int_0^t f_2(t-s)f_1(\omega(s),s) ds&&(\mbox{for}\ \omega\in D([0,\infty),E)).
\end{align*}
\end{remark}
\begin{lemma}\label{lem:skorokhod}
Let $f:E\to \R$ be continuous
 and bounded, 
then the functional 
\begin{align*}
F_3(\omega)&=\int_0^\infty e^{-\delta s}f(\omega(s)) ds
\qquad(\mbox{for}\ \omega\in D([0,\infty),E))
\end{align*}
is Skorokhod continuous.
\end{lemma}
\begin{proof} Let $\sigma$ denote the Skorokhod metric on $D([0,\infty),E)$.
Let $\varepsilon>0$. There exists $t>0$ such that $\int_t^\infty e^{-\delta s}\|f\|_\infty ds <\frac{\veps}{4}$. By Skorokhod continuity of $F_2$ 
there exists an $\eta>0$ such that for all $\omega_1,\omega_2\in D([0,\infty),E)$ it
holds that, if $\sigma(\omega_1,\omega_2)<\eta$ then $\left|\int_0^te^{-\delta
s}f(\omega_1(s)) ds-\int_0^te^{-\delta s}f(\omega_2(s))
ds\right|<\frac{\veps}{2}$. Therefore,
\begin{align*}
|F_3(\omega_1)-F_3(\omega_2)|&=\left|\int_0^\infty e^{-\delta s}f(\omega_1(s)) ds-\int_0^\infty e^{-\delta s}f(\omega_2(s)) ds\right|\\
&\le \left|\int_0^te^{-\delta s}f(\omega_1(s)) ds-\int_0^te^{-\delta s}f(\omega_2(s)) ds\right|+2 \int_t^\infty e^{-\delta s}\|f\|_\infty ds<\veps.
\end{align*} 
\end{proof}

We remark that a function $f:E\to\R$ is continuous if and only if $f_k:E_k\to\R$ is continuous for all $k$. In particular, 
every indicator function of a component $\{k\}\times E_k$ is continuous.

We are in the position to show that cost functionals indeed 
commute with weak limits of PDMPs.
\begin{lemma}
Let $X$ be a PDMP and $(X^n)_{n\in\N}$ be a sequence of PDMPs on the same state space $E$ 
and with the same cemetery $E^c$, and let $\ell:E\to \R$ and $\Psi:E\to \R$
be a running reward function and a terminal cost function, respectively.
Assume that both $\ell$ and $\Psi$ are continuous and bounded. 
Assume further that $X^n_0=x$ for all $n\in \N$
and $X_0=x$, and  $X^n\stackrel{d}{\to}X$ in $D([0,\infty),E)$.

Then  
\begin{align*}
\E_{x}\left(\int_0^\tau e ^{-\delta t} \ell(X^n_t) dt+e^{-\delta\tau}\Psi(X^n_\tau)\right)
\to
\E_{x}\left(\int_0^\tau e ^{-\delta t} \ell(X_t) dt+e^{-\delta\tau}\Psi(X_\tau)\right)
\end{align*}
as $n\to \infty$.
\end{lemma}

\begin{proof}
Recall that $\ell\equiv 0$  on $E^c$, and $\Psi\equiv 0$  on $E\backslash 
E^c$, so that 
$\int_0^\infty e^{-\delta s}\ell(\omega(s)) ds=\int_0^\tau e^{-\delta
s}\ell(\omega(s)) ds$ and $\int_0^\infty \delta e^{-\delta s}\Psi(\omega(s))ds=\int_\tau^\infty \delta e^{-\delta s}\Psi(\omega(s))ds$. 
Thus by Lemma \ref{lem:skorokhod} the mappings  $\omega\mapsto \int_0^\tau e^{-\delta
s}\ell(\omega(s)) ds$ and $\omega\mapsto \int_\tau^\infty\delta e^{-\delta
s}\Psi(\omega(s)) ds$ are  Skorokhod continuous. 

Moreover, if $\omega$ is a path of the 
PDMPs, then it holds that $\omega(s)=\omega(\tau)$
for all $s\ge\tau$, such that 
$\int_\tau^\infty \delta e^{-\delta s}\Psi(\omega(s))ds
=e^{-\delta \tau}\Psi(\omega(\tau))$.
This completes the proof.
\end{proof}
Also, finite time ruin probabilities, i.e., the probability of the PDMP
reaching the cemetery before a given time horizon $t$, 
commute with weak limits, as we show next.
\begin{lemma}
Let $X$ be a PDMP and $(X^n)_{n\in\N}$ be a sequence of PDMPs on the same state space $E$ 
and with the same cemetery $E^c$. 
Assume further that $X^n_0=x$ for all $n\in \N$
and $X_0=x$, and  $X^n\stackrel{d}{\to}X$ in $D([0,\infty),E)$.

Then $\lim_{n\to\infty }\P_x(X^n_t\in E^c)=\P_x(X_t\in E^c)$ for all $t\ge 0$.
\end{lemma}

\begin{proof}
Consider a functional of the same form as $F_1$ in Remark \ref{rem:skorokhod},
with $f_1=\1_{E^c}$. Since the cemetery is the union of only entire $(\{k\}\times E_k)$, 
and is therefore a union of connected components of $E$,
the indicator function of the cemetery is continuous.
Therefore if we define 
$\psi(x,t)=\P_x(\tau\le t)=\P_x(X_t\in E^c)=\E_x(\1_{E^c}(X_t))$ and 
$\psi^n(x,t)=\P_x(\tau^n\le t)=\P_x(X^n_t\in E^c)=\E_x(\1_{E^c}(X^n_t))$, $n\in \N$, we have
$\lim_{n\to\infty}\psi^n(x,t)=\psi(x,t)$ for all $x\in E$ and for all $t\ge 0$.
\end{proof}

The following theorem specifies conditions under which Theorem \ref{th:kallenberg} is applicable in the PDMP setting. 
\begin{theorem}\label{thm:convPDMP}
Let $X$ be a Feller PDMP with local characteristics $(\phi,\lambda,Q)$ and let $X^n$, $n\in \N$, be Feller PDMPs with local characteristics $(\phi^n,\lambda^n,Q^n)$.
Further, let the following assumptions hold: 
\begin{enumerate}[(i)]
\item
$\odecoeff^n\to \odecoeff$ and $\lambda^n\to \lambda$ as $n\to\infty$, 
uniformly in $x\in E$,
\item for all $f\in C_b^\infty(E,\R)$, 
\begin{align}\label{eq:convMeasure}
\lim_{n\to\infty}\sup_{x\in E}\left|\int_E f(y) Q^n(dy,x)-\int_E f(y) Q(dy,x)\right|= 0,
\end{align}
\item $X_0^n\stackrel{d}{\to} X_0$ in $E$.
\end{enumerate}
Then $X^n\stackrel{d}{\to} X$ in $D([0,\infty),E)$.
\end{theorem}
\begin{proof}
Let $\mathcal{D}(\cA^n)$, $n\in\N$, and $\mathcal{D}(\cA)$ be the domains of 
the generators $\cA^n$, $n\in\N$, and $\cA$, corresponding to $X^n$ and $X$, respectively. For $f^n\in\mathcal{D}(\cA^n)$ we have
\begin{align*}
\mathcal{A}^n f^n(x)&=\calX^n f^n(x)+\lambda^n(x)\int_E (f^n(y)-f^n(x))Q^n(x,dy),\\
(\calX^n f^n)(x)&=(\odecoeff^n)(x)\cdot \nabla (f^n)(x).
\end{align*}
By Proposition \ref{prop:core}, $D=C_b^\infty(E,\R)$ is a core for all generators involved.
For every $f\in D$ we set $f^n=f$ for all $n\in\N$, such that trivially
$f^n\to f$ as $n\to \infty$. 
Next, observe that we have for all $n\in \N$,
\begin{align}
\nonumber&\vert\mathcal{A}^n f(x)-\mathcal{A}f(x)\vert\leq
\vert \odecoeff^n(x)\cdot\nabla f(x) - \odecoeff(x)\cdot\nabla f(x) \vert\\
\nonumber&+\left\vert \lambda^n(x)\int_E (f(y)-f(x))Q^n(dy,x)-\lambda(x)\int_E (f(y)-f(x))Q(dy,x)\right\vert\\
\nonumber&=\vert (\odecoeff^n(x)-\odecoeff(x))\cdot\nabla f(x)\vert
+\left\vert\lambda^n(x)\int_E (f(y)-f(x))Q^n(dy,x)-\lambda(x)\int_E (f(y)-f(x))Q(dy,x)\right\vert\\
\label{eq:firstsecondterm1}&\leq \|\odecoeff^n-\odecoeff\|_\infty
\|\nabla f\|_\infty+\|f\|_\infty\left\vert \lambda^n(x) \int_E Q^n(dy,x)-\lambda(x)\int_E Q(dy,x)\right\vert\\
\label{eq:thirdterm1}&+\left\vert \lambda^n(x) \int_E f(y) Q^n(dy,x)-\lambda(x)\int_E f(y) Q(dy,x)\right\vert.
\end{align}
Since $Q^n$, $n\in\N$, and $Q$ are probability measures on $(E,\mathcal{B}(E))$, and since, by assumption, $\odecoeff^n\to \odecoeff$ and $\lambda^n\to \lambda$ uniformly in $x\in E$, 
the terms in \eqref{eq:firstsecondterm1} converge to zero. The term in \eqref{eq:thirdterm1} can be estimated as follows,
\begin{align*}
&\left\vert \lambda^n(x) \int_E f(y) Q^n(dy,x)-\lambda(x)\int_E f(y) Q(dy,x)\right\vert\\
&\leq \|\lambda^n\|_\infty\left|\int_E f(y) Q^n(dy,x)-\int_E f(y) Q(dy,x)\right|+\left|\int_E f(y)Q(dy,x)\right|\|\lambda^n-\lambda\|_\infty.
\end{align*}
The latter expression tends to zero, since for all $x\in E$ it was assumed that \eqref{eq:convMeasure} holds, and since $\lambda^n\to \lambda$ uniformly in $x\in E$.
Thus, Item \eqref{it:kallenberg1} of Theorem \ref{th:kallenberg} holds. This implies that Item \eqref{it:kallenberg4} of Theorem \ref{th:kallenberg} holds.
The latter is equivalent to the assertion of this theorem.
\end{proof}
\begin{remark}
Note that in the Feller case we can move to
another external state only due to a purely random jump, i.e., a jump determined by 
$Q^n$ for $n\in\N$ or $Q$. Therefore, if we assume
uniform convergence of the local characteristics across all state components, and
in particular also $Q^n \to Q$ in the sense of \eqref{eq:convMeasure}, then the result of Theorem \ref{thm:convPDMP} still holds.
\end{remark}
Since uniform convergence of the local characteristics  and the assumption that
$t^\ast (x)=\infty$ are essential in the proof of Theorem \ref{thm:convPDMP}, we need an alternative argument for situations
with an active boundary or for situations in
which a smooth approximation fails. A prototypical univariate example for both
cases is a drift of the form $g(x)=c\,\1_{\{x\le b\}}$ for some $b\in \R$. Here one
faces either a discontinuity or a subdivision of $\R$ into two state components, i.e.,
$\R=\{x\in\R\colon x\le b\}\cup\{x\in\R\colon x> b\}$, with a continuous drift on each component.
For a specific example, we find a method for dealing with this
particular situation in the next section.

\section{Application to the Cram\'er-Lundberg model with loan}\label{sec:DE}
In this section we apply our smoothing technique to the example presented in Subsection \ref{subsubsec:CL-loan} and
calculate the quantity of interest using different numerical integration methods.
In this setup, $\phi_1$ solves the ODE
$\frac{\partial}{\partial t}\phi_1(y,t)=\odecoeff_1(\phi_1(y,t))$
$\forall y\in E_1$ and $\forall t\in \R$, with
\begin{align*}
\odecoeff_1(y)=\begin{cases}
c &\mbox{if}\ y\in(0,\infty),\\
c+\rho y&\mbox{if}\ y\in(-\frac{c}{\rho},0],\\
0 &\mbox{if}\ y\in(-\infty,-\frac{c}{\rho}].
\end{cases}
\end{align*}

In the setup of Subsection \ref{subsubsec:CL-loan}, 
the quantity of interest is the expected value of discounted future dividend payments until the time of ruin.
The cemetery $E^c$ is given by $E^c=(\{2\}\times E_2)\cup(\{3\}\times E_3)$,
the running reward $\ell$ is given by $\ell_1\equiv 0, \ell_4\equiv c$, and the terminal cost  is
$\Psi(x)=0$ for $x\in E^c$.
For $x\in E$, $t\ge0$, let
\begin{align*}
L(t,x)=\int_0^t e^{-\delta s}\ell(\phi(s,x))\,ds.
\end{align*}
Since $\odecoeff$ is not differentiable in $0$ and $t^*(x)<\infty$ for all $x=(1,y)$ with $y\in E_1$, 
we replace $\odecoeff$ by a smoothed version and we also modify $\ell$ accordingly.
For $\varepsilon>0$ let
\begin{align*}
\odecoeff_1^{\veps}(y)=
\begin{cases}
c	&\mbox{if}\  y\in(\veps,b-\veps],\\
\frac{c (b-y)^3 \left(15 \veps(y-b)+6 (b-y)^2+10 \veps^2\right)}{\veps^5}	&\mbox{if}\  y \in(b-\veps,b),\\
c+\rho y	&\mbox{if}\  y\in(-\frac{c}{\rho},-\veps),\\
c+\frac{\rho(y+3 \veps )(y-\veps )^3}{16 \veps ^3}	&\mbox{if}\  y\in[-\veps,\veps],\\
0			&\mbox{if}\  y\in(-\infty,-\frac{c}{\rho}]\cup[b,\infty).
\end{cases}
\end{align*}
Observe that $\odecoeff_1^\varepsilon \in C^2(\R)$, that $\lim_{y\nearrow b}\odecoeff_1^\veps (y)=0$ and that $\odecoeff_1^\veps\geq 0$.
For $\varepsilon>0$ define the PDMP $X^\veps$ so that for all $y\in \R$ its flow $\phi_1^\veps(\cdot,y)$ is the solution to the ODE 
$\frac{d}{dt}\phi_1^\veps(t,y)=\odecoeff^\veps(\phi_1^\veps(t,y))$ with $\phi_1^\veps(0,x)=x$.
Apart from that all specifications are the same as for the original PDMP $X$.
In addition, we replace
$\ell_1$ by
\begin{align*}
\ell_1^\veps(y)=c\,\sheavi\!\left(\frac{y-b+\veps}{\veps}\right),
\end{align*}
where $h$ can be chosen as in \eqref{eq:sheavi} and we define
\begin{align*}
L^\veps(t,x)=\int_0^t e^{-\delta s}\ell^\veps(\phi^\veps(s,x))\,ds.
\end{align*} 
We aim at computing $\mathcal{G}^{i-1}\mathcal{H}$ for the smoothed process, 
in order to observe how \eqref{eq:iterate} simplifies in this example.
By the definition of the cemetery, $\mathcal{G}^{i-1}\mathcal{H}(x)=0$ for all $x=(k,z)\in E$ with 
$k\in\{2,3\}$. 
For $x=(1,z)$ with $z\in E_1$, any jumps bigger than $z+c/\rho$ lead to the cemetery, so we only need to integrate
over jump sizes up to $z+c/\rho$. Thus, we get that
\begin{align*}
\mathcal{G}V(x)
=\mathcal{G}V((1,z))
&=\int_0^{\infty}f_W(t,x)e^{-\delta t}\int_E V(y)Q(dy,\phi(x,t))dt\\
&=\int_0^{\infty}f_W(t,x)e^{-\delta t}\int_0^{z+c/\rho} V((1,z-y))dF_Y(y)dt\,.
\end{align*}
Moreover, since $\lambda$ is constant on $E_1$ it holds for all $x=(1,z)$ with $z\in E_1$, $t\ge 0$ that 
$f_W(t,x)=\lambda_Ne^{-\lambda_Nt}$, where $\lambda_N$ is as in Section \ref{subsubsec:C-L}.
For $x=(1,z)$ with $z\in E_1$ we get 
\begin{multline}\label{eqn:integral-de1}
\mathcal{G}^{i-1}\mathcal{H}(x)=
\int_{t_1=0}^\infty\lambda_N e^{-(\lambda_N+\delta)t_1}\int_{y_1=0}^{\chi_{1^-}+\frac{c}{\rho}}\cdots
\int_{t_{i-1}=0}^\infty\lambda_N e^{-(\lambda_N+\delta)t_{i-1}}\int_{y_{i-1}=0}^{\chi_{(i-1)^-}+\frac{c}{\rho}}\\
\int_{t_i=0}^\infty\lambda_N e^{-\lambda_N t_i}L^\veps(t_i,\chi_{(i-1)})dt_i\,dF_Y(y_{i-1})dt_{i-1}\cdots dF_Y(y_1)dt_1,
\end{multline}
where the functions $\chi_{j^-},\chi_{j}$, $j=1,2,\ldots,i-1$ 
solve $\chi_{j^-}=\phi_1^\veps(t_j,\chi_{j-1})$ and $\chi_{j}=\chi_{j^-}-y_j$.

Thus $\chi_{j^-}$ depends on $t_1,\ldots,t_j$ and $y_1,\ldots,y_{j-1}$, whereas 
$\chi_{j}$ depends on $t_1,\ldots,t_j$ and $y_1,\ldots,y_j$. However, this
dependence has been suppressed in \eqref{eqn:integral-de1} for the sake
of readability.
\begin{assumption}\label{ass:density}
 The jump distribution admits a density $f_Y=F_Y'$, with $f_Y\in C_0^2$.
\end{assumption}
In what follows, suppose that Assumption \ref{ass:density} holds.
A variable substitution $t_j=-\ln(v_j)$ and $y_j=(\chi_{j^-}+\frac{c}{\rho})z_j$,
where $v_j\in[0,1],\;z_j\in[0,1]$, 
$\hat{\chi}_j(v_1,\ldots,v_j,z_1,\ldots,z_{j})=\chi_{j^-}(t_1,\ldots,t_j,y_1,\ldots,y_{j})$.
We then put
$$\nu(v_1,\ldots,v_i,z_1,\ldots,z_{i-1})=L^\veps(-\ln (v_i),\hat{\chi}_{i-1}(v_1,\ldots,v_{i-1},z_1,\ldots,z_{i-1})),$$
which leads to
\begin{multline}\label{eq:DE_Integrand}
\mathcal{G}^{i-1}\mathcal{H}(x)=\int_{v_1=0}^1\cdots\int_{v_i=0}^1\int_{z_1=0}^1\cdots\int_{z_{i-1}=0}^1
\lambda_N^i \left[\prod_{j=1}^{i-1}v_j^{\delta+\lambda_N-1}\right] v_i^{\lambda_N-1}\nu(v_1,\ldots,v_i,z_1,\ldots,z_{i-1})\\
\times\left[\prod_{j=1}^{i-1}f_Y\!\left(z_j\left(\hat{\chi}_j+\frac{c}{\rho}\right)\right)\left(\hat{\chi}_j+\frac{c}
{\rho}\right)dz_j\right]\prod_{j=1}^i dv_j.
\end{multline}
Due to the recursive structure of the functions $\hat{\chi}_1,\hat{\chi}_2,\ldots,\hat{\chi}_{i-1}$, the  
Jacobi matrix of the substitution has lower triangular shape, 
such that its determinant is the 
product of the diagonal elements.
For being able to reasonably apply (\ref{eqpausva}) we need to bound 
the Hessian of the integrand.
If for example the jump size distribution is the Gamma distribution with parameters $\alpha,\beta>0$, 
i.e.,~$dF_Y(y)=\frac{y^{\alpha-1}\beta^\alpha e^{-\beta y}}{\Gamma(\alpha)}dy$, then this boils down to the condition 
$\beta\geq 3$ and $\delta+\lambda>3$, 
which implies that the integrand is bounded in $\bf{0}$. In the original problem statement this corresponds to an additional integrability condition on 
the jump size distribution.
Finally, for  $x\in E$ of the form $x=(4,b)$ we have 
\[
\mathcal{G}^{i-1} \mathcal{H}((4,b))=\int_0^\infty \lambda e^{-\lambda t}e^{-\delta t}
\int_0^{b+c/\rho} \mathcal{G}^{i-2} \mathcal{H}((1,b-y))dF_Y(y)dt\,.
\] 
\begin{remark}
In Section \ref{sec:FellerCase} the stability, with respect to the smoothing parameter $\varepsilon$ of the considered functional
of the process, is dealt with in
a fairly general setting. Unfortunately, because of the discontinuity of the drift $\odecoeff$ in the present example, we cannot achieve uniform convergence
of the smoothed drift around the barrier level $b$, whereas point-wise convergence is achieved.
\end{remark}
\begin{theorem}
In the setup of this section, the following assertion holds true. There exists  $C>0$ 
such that $\|V^\veps- V\|_\infty \le C \veps$.
\end{theorem}
\begin{proof}
Recall that 
\begin{align*}
V(x)=\E_x\left(\int_0^\tau e^{-\delta s} \ell(X_s)ds\right) \quad \text{ and } 
\quad V^\veps(x)=\E_x\left(\int_0^{\tau^\veps} e^{-\delta s} \ell^\veps(X_s^\veps)ds\right),
\end{align*}
where $\tau=\inf\{t\geq 0 \colon X_t\in E^c\}$ and 
$\tau^\veps=\inf\{t\geq 0 \colon X_t^\veps\in E^c\}$.

It is readily checked that 
 $\sup_{y\in(-c/\rho,b-\veps)}\vert \odecoeff_1(y)-\odecoeff^\veps_1(y)\vert\leq 3\veps \rho/16$ 
and that $\vert \odecoeff_1(y_1)-\odecoeff_1(y_2)\vert \leq \rho\vert y_1-y_2\vert$
for all $y_1,y_2\in (-c/\rho,b-\veps)$. 
Hence  we get from \cite[Theorem 8, p.~111]{Kamke1964} that
\begin{align*}
\vert \phi^\veps_1(t,y)- \phi_1(t,y)\vert <\frac{3\veps}{16}\left(e^{\rho t}-1\right)
\end{align*}
for all $y\in(-c/\rho,b-\veps)$ and for all $t\in[0,\min\{\theta^\veps_{b-\veps},\tilde{\theta}_{b-\veps}\}]$, where 
\begin{align*}
\theta^\veps_{b-\veps}&=\inf\{t\geq 0 \colon \phi^\veps_1(t,y)\geq b-\veps\}\,,\\
\tilde{\theta}_{b-\veps}&=\inf\{t\geq 0 \colon \phi_1(t,y)\geq b-\veps\}\, ,\\
\tilde{\theta}_{b}&=\inf\{t\geq 0 \colon \phi_1(t,y)\geq b\}\, .\\
\end{align*}
Since $\odecoeff^\veps_1$ and $\odecoeff_1$ coincide on  
$(-c/\rho,b-\veps)\backslash(-\veps,\veps)$ and 
$\odecoeff^\veps_1\ge \odecoeff_1\ge0$, we can refine this estimate
to get 
\begin{align*}
\vert \phi^\veps_1(t,y)- \phi_1(t,y)\vert <\frac{3\veps}{16}\left(e^{\rho C(\veps)}-1\right),
\end{align*}
for all $t\in [0,\min\{\theta^\veps_{b-\veps},\tilde{\theta}_{b-\veps}\}]$, where 
$C(\veps)\in [0,\infty)$ is the time needed for the trajectory $\phi_1(\cdot,y)$ 
to cross $(-\veps,\veps)$. 
Note that $\odecoeff^\veps_1\geq\odecoeff_1\geq 0$
yields $\phi^\veps_1(t,y)\geq \phi_1(t,y)$ for 
all $t\in [0, \tilde{\theta}_{b-\veps}]$, and hence $\tilde{\theta}_{b-\veps}\ge \theta^\veps_{b-\veps}$.
For $t\geq\min\{\theta^\veps_{b-\veps},\tilde{\theta}_{b-\veps}\}=\theta^\veps_{b-\veps}$ 
we have by construction
that $\vert \phi^\veps_1(t,y)- \phi_1(t,y)\vert\leq \veps$. 
In total we get
\begin{align}\label{eq:pathestimate}
\vert \phi^\veps_1(t,y)- \phi_1(t,y)\vert\leq
\veps \max\left(1,\frac{3}{16}\left(e^{\rho C(\veps)}-1\right)\right).
\end{align}
Since $\lim_{\veps\to 0}C(\veps)\to 0$, it holds that
$\vert \phi^\veps_1(t,y)- \phi_1(t,y)\vert\leq \veps$ for sufficiently small $\veps>0$.

Recall that $T_1$ 
is the time of the first jump of $X$
conditional on $X_0=(1,y)$. 
Since the jump intensity is constant on $E_1$, 
$T_1$ is exponentially distributed with intensity $\lambda_N$.
Hence, we can write
\begin{align*}
V((1,y))&=\E_{(1,y)}\left(L(T_1,(1,y))+e^{-\delta T_1}V(X_{T_1})\right)\\
&=\int_0^\infty \lambda_Ne^{-\lambda_N s}\left(L(s,(1,y))+e^{-\delta s}\!\int_E V(x_1)Q(dx_1,\phi_1(s,y))\right)ds\\
&=\int_0^\infty \lambda_Ne^{-\lambda_N s}L(s,(1,y))ds+\int_0^\infty \!\lambda_Ne^{-(\lambda_N+\delta) s}\!\int_E V(x_1)Q(dx_1,\phi_1(s,y))\,ds\,,
\end{align*}
and analogously for $V^\veps$. We write $V((1,y))=V_1(y)$ and $V^\veps((1,y))=V^\veps_1(y)$ for $y\in E_1$.
Therefore,
\begin{align}
\nonumber
|V_1(y)-V^\veps_1(y)|
&\leq \int_0^\infty \lambda_Ne^{-\lambda_N s}|L(s,(1,y))-L^\veps(s,(1,y))|ds\\
\nonumber&+
\int_0^\infty \!\lambda_Ne^{-(\lambda_N+\delta) s}\left\vert\int_E V(x_1)Q(dx_1,\phi_1(s,y))-\int_E V^\veps(x_1)Q(dx_1,\phi^\veps_1(s,y))\right\vert\,ds\\
\nonumber&\leq \int_0^\infty \lambda_Ne^{-\lambda_N s}|L(s,(1,y))-L^\veps(s,(1,y))|ds\\
\nonumber&+
\int_0^\infty \!\lambda_Ne^{-(\lambda_N+\delta) s}\left\vert\int_E V(x_1)Q(dx_1,\phi_1(s,y))-\int_E V(x_1)Q(dx_1,\phi^\veps_1(s,y))\right\vert\,ds\\
\label{eq:VVeps}&+
\int_0^\infty \!\lambda_Ne^{-(\lambda_N+\delta) s}\left\vert\int_E V(x_1)Q(dx_1,\phi^\veps_1(s,y))-\int_E V^\veps(x_1)Q(dx_1,\phi^\veps_1(s,y))\right\vert\,ds\,.
\end{align}
For $x=(1,y)$ and $t\ge 0$ it holds that $L^\veps(s,x)=0$ for 
$s\le \theta^\veps_{b-2 \veps}$, and  
\begin{align*}
L^\veps(s,x)&=\int_0^s e^{-\delta r}\ell^\veps(\phi^\veps(r,x))dr
=c \int_0^s e^{-\delta r}\sheavi((\phi^\veps(r,x)-b+\veps)/\veps) dr
\le c \int_{\theta^\veps_{b-2\veps}}^s  e^{-\delta r} dr
\end{align*}
for $s\ge \theta^\veps_{b-2 \veps}$.
On the other hand, we have that, for 
$x=(1,y)$ and $s\ge 0$, $L(s,x)=0$ for 
$s\le \tilde{\theta}_b$, and  
\begin{align*}
L(t,x)&=
c \int_{\tilde{\theta}_b}^s  e^{-\delta r} dr
\end{align*}
for $s> \tilde{\theta}_b$.
Using $\phi^\veps_1(s,y)\geq \phi_1(s,y)$ for all $s\ge 0$,  we get
$\tilde{\theta}_b\ge \theta^\veps_{b-2 \veps}$, such that 
\begin{align*}
|L^\veps(s,(1,y))-L(s,(1,y))|
=L^\veps(s,(1,y))-L(s,(1,y))
\le c \int_{\theta^\veps_{b-2\veps}}^{\tilde{\theta}_b}  e^{-\delta r} dr
\end{align*}
for all $t\ge 0$. Hence,
\begin{align*}
\int_0^\infty \lambda_Ne^{-\lambda_N s}|L(s,(1,y))-L^\veps(s,(1,y))|ds
\le c\int_{\theta^\veps_{b-2\veps}}^{\tilde{\theta}_b}  e^{-\delta r} dr
\le c(\tilde{\theta}_b-\theta^\veps_{b-2\veps})\,.
\end{align*}
Now $\tilde{\theta}_b-\theta^\veps_{b-2\veps}\leq \frac{b-(b-2\veps-\veps C_1(\veps))}{c}=\veps\frac{2+ C_1(\veps)}{c}$, where
$C_1(\veps)=\max\left(1,\frac{3}{16}\left(e^{\rho C(\veps)}-1\right)\right)$,
see \eqref{eq:pathestimate}.
With this, the first term in \eqref{eq:VVeps} can be estimated by
\begin{align}\label{eq:rhs1}
\int_0^\infty \lambda_Ne^{-\lambda_N s}|L(s,(1,y))-L^\veps(s,(1,y))|ds
\le \veps(2+C_1(\veps))\,.
\end{align}
Next, observe that (we remind the reader that the states $x\in E$ are denoted by $x=(k,y)$, which is why in the following 
the terms $y_1, y_2$ are not to be confused with the integration variables 
$y_j$ used in and below \eqref{eqn:integral-de1}), 
\begin{align*}
\lefteqn{\left\vert\int_E V(x_1)Q(dx_1,(1,y_1))-\int_E V(x_1)Q(dx_1,(1,y_2))\right\vert}\\
&=
\left\vert\int_0^{y_1+c/\rho} V_1(y_1-z)f_Y(z)dz-\int_0^{y_2+c/\rho} V_1(y_2-z)f_Y(z)dz\right\vert\\
&\le 
\left\vert\int_{\min(y_1,y_2)+c/\rho}^{\max(y_1,y_2)+c/\rho} V_1(z)f_Y(z)dz\right\vert
\le \|V_1\|_\infty\|f_Y\|_\infty|y_1-y_2|.
\end{align*}
Combining this with \eqref{eq:pathestimate}, we can estimate the second term in \eqref{eq:VVeps} by
\begin{align}\label{eq:rhs2}
\lefteqn{\int_0^\infty \!\lambda_Ne^{-(\lambda_N+\delta) s}\left\vert\int_E V(x_1)Q(dx_1,\phi_1(s,y))-\int_E V^\veps(x_1)Q(dx_1,\phi^\veps_1(s,y))\right\vert\,ds}\\
&\le \frac{\lambda_N}{\lambda_N+\delta}
 \|V_1\|_\infty\|f_Y\|_\infty\sup_{s\ge 0}|\phi_1(s,y)-\phi^\veps_1(s,y)|
\le \frac{\lambda_N}{\lambda_N+\delta}
 \|V_1\|_\infty\|f_Y\|_\infty
\veps \max\left(1,\frac{3}{16}\left(e^{\rho C(\veps)}-1\right)\right).
\end{align}
Furthermore, since
\begin{align*}
\left\vert\int_E V(x_1)Q(dx_1,(1,y_2))-\int_E V^\veps(x_1)Q(dx_1,(1,y_2))\right\vert
\le 
\left\| V_1-V^\veps_1\right\|_\infty,
\end{align*}
the third term in \eqref{eq:VVeps} can be estimated as follows,
\begin{align}\label{eq:rhs3}
\int_0^\infty \!\lambda_Ne^{-(\lambda_N+\delta) s}\left\vert\int_E V(x_1)Q(dx_1,\phi^\veps_1(s,y))-\int_E V^\veps(x_1)Q(dx_1,\phi^\veps_1(s,y))\right\vert\,ds
\
\le \frac{\lambda_N}{\lambda_N+\delta}\|V_1-V^\veps_1\|_\infty\,.
\end{align}
Taking the supremum over $y\in E_1$ in \eqref{eq:VVeps} and 
using \eqref{eq:rhs1}, \eqref{eq:rhs2}, and \eqref{eq:rhs3} we obtain that
\begin{align*}
\|V_1-V^\veps_1\|_\infty\le C \varepsilon +\frac{\lambda_N}{\lambda_N+\delta}\|V_1-V^\veps_1\|_\infty
\end{align*}
for some constant $C$ and for sufficiently small $\veps$.
Thus, 
\begin{align*}
\frac{\delta}{\lambda_N+\delta}
\|V_1-V^\veps_1\|_\infty\le C \varepsilon,
\end{align*}
which completes the proof.
\end{proof}

\subsection{Numerical experiment}\label{subsec:numerical}

We now solve the example presented above numerically.
We set the following parameter values. The initial value of the PDMP $x_0=0$, the premium income rate $c=5$, the credit rate $\rho=0.05$, the intensity of the Poisson process $\lambda=4$, the jump size distribution is for all $x\in[0,\infty)$ given by $F_Y(x)=1-e^{-\alpha x}$ with $\alpha=1$, and the discount rate $\delta=0.02$. With this, the optimal dividend threshold according to \cite{DassiosEmbrechts1989} is $b=3.24289$. Furthermore, we set the smoothing parameter $\varepsilon=0.01$. For computing the  flow it is enough to solve the corresponding
ODE once and to store the solution for repeated use.

We implemented Monte Carlo (random), quasi-Monte Carlo with the Sobol' sequence (Sobol), and quasi-Monte Carlo with a scrambled version of the Halton sequence (scrambled Halton), where scrambling refers to a permutation of digits (see, e.g., \cite{owen98}).
The Sobol' point generator we used was taken from Frances Y.~Kuo's homepage \cite{kuohomepage} and is based on \cite{joe2008}.

The reference solution was calculated using Monte Carlo with
$M=5000\cdot 2^{10}=5120000$ sample paths
and $d=1024$, meaning that the maximum number of jumps we allow for is 512.
In our plots we show the results plotted over an increasing number of integration nodes $M\in\{50\cdot 2^j\colon 1\le j \le 16\}$.

Figure \ref{fig:stdder} shows the estimated standard deviation (root mean square error) of the estimation, which is calculated by using 50 repetitions with randomly shifted versions of our integration nodes.
\begin{figure}[ht]
\begin{center}
\includegraphics[width=\textwidth]{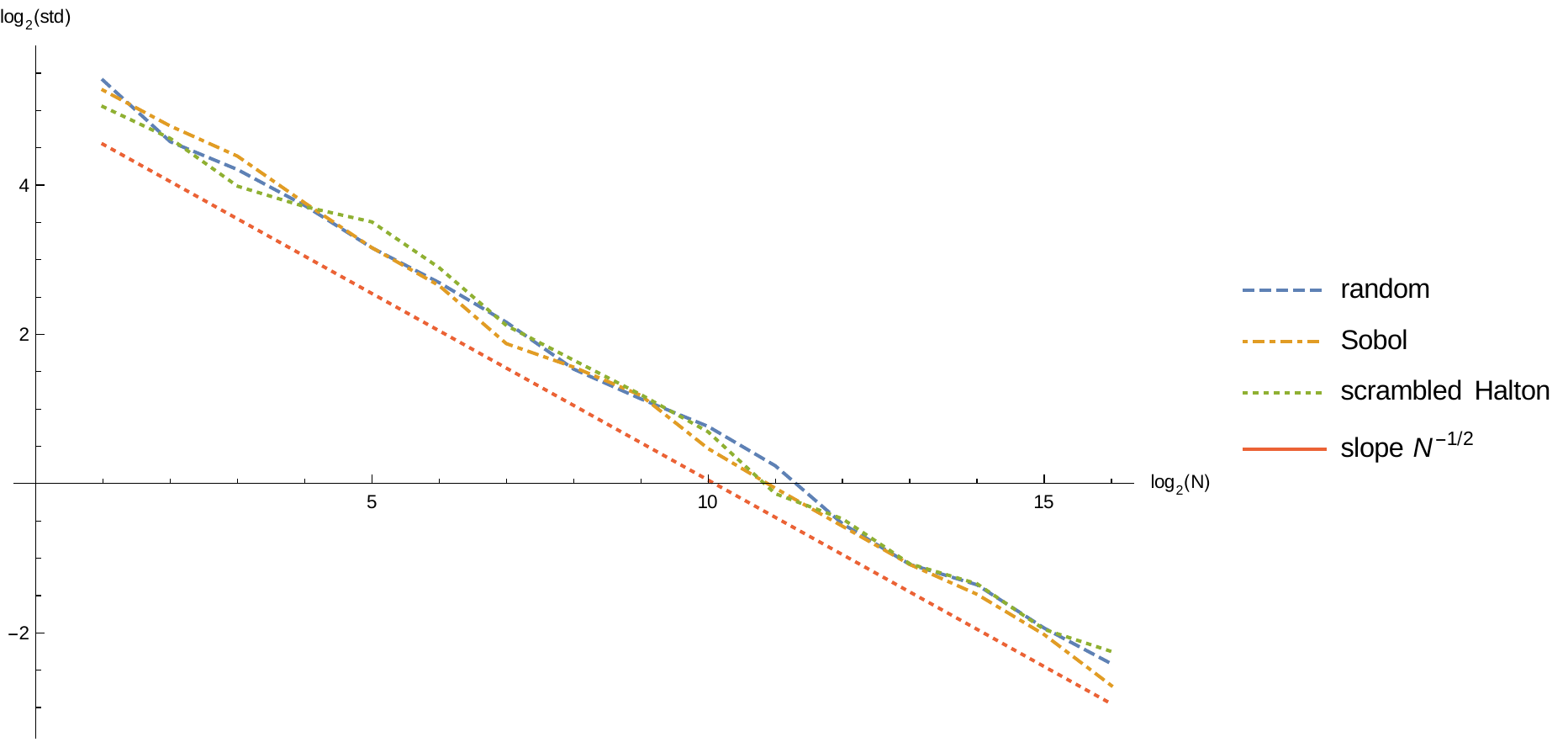}
\caption{The estimated standard deviation of the estimation.}\label{fig:stdder}
\end{center}
\end{figure}

\section*{Acknowledgements}

The authors would like to thank an anonymous referee for useful comments on how to improve the presentation of the results.





\vspace{2em}
\centerline{\underline{\hspace*{17cm}}}

\noindent Peter Kritzer\\
Johann Radon Institute for Computational and Applied Mathematics (RICAM), Austrian Academy of Sciences, Altenbergerstraße 69, 4040 Linz, Austria\\
peter.kritzer@ricam.oeaw.ac.at\\

 \noindent Gunther Leobacher \\
 Institute for Mathematics and Scientific Computing, University of Graz, Heinrichstra\ss{}e 36, 8010 Graz, Austria\\
 gunther.leobacher@uni-graz.at\\

\noindent Michaela Sz\"olgyenyi \\
Institute for Statistics, University of Klagenfurt, Universit\"atsstra\ss{}e 65-67, 9020 Klagenfurt, Austria and\\
Seminar for Applied Mathematics and RiskLab Switzerland, ETH Zurich, R\"amistrasse 101, 8092 Zurich, Switzerland\\
michaela.szoelgyenyi@aau.at\\

\noindent Stefan Thonhauser \Letter\\
Institute for Statistics, Graz University of Technology, Kopernikusgasse 24/III, 8010 Graz, Austria\\
stefan.thonhauser@math.tugraz.at

\end{document}

%% file: pdmp_gen_klein7.pspdftex
\begin{picture}(0,0)%
\includegraphics{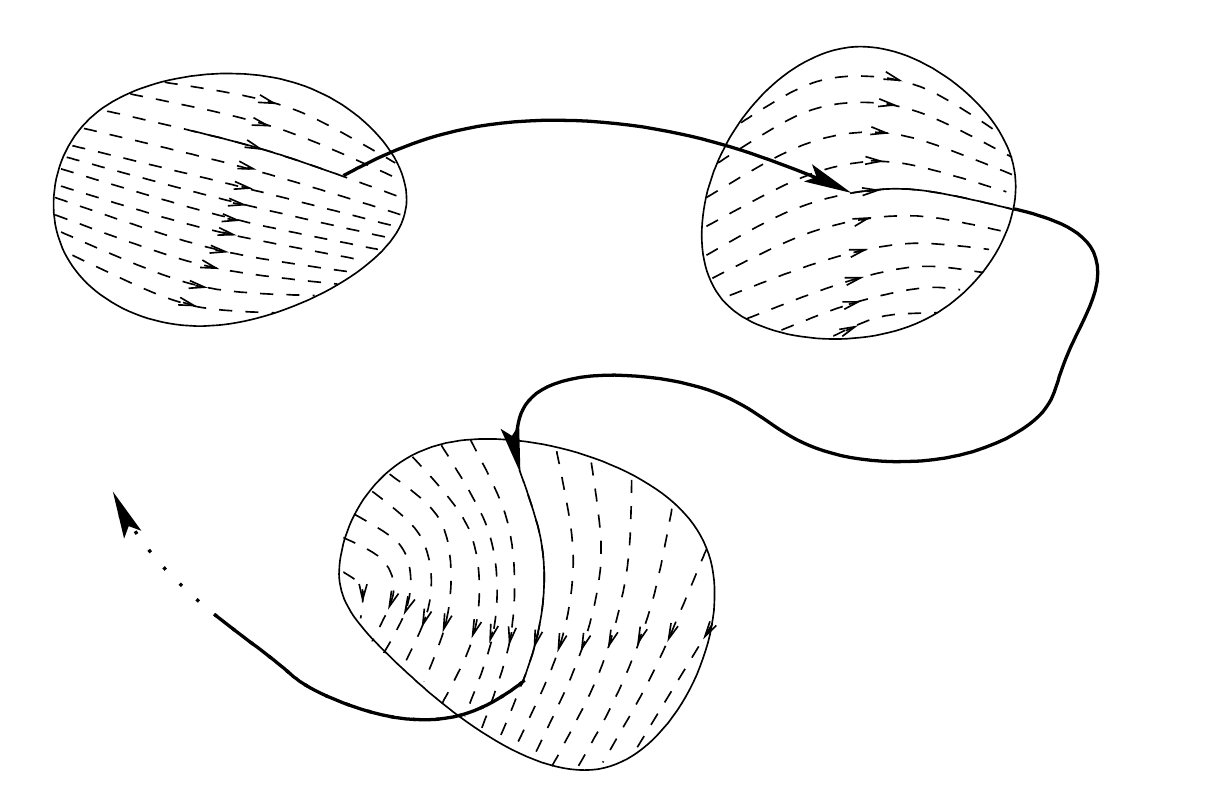}%
\end{picture}%
\setlength{\unitlength}{4144sp}%
\begingroup\makeatletter\ifx\SetFigFont\undefined%
\gdef\SetFigFont#1#2#3#4#5{%
  \reset@font\fontsize{#1}{#2pt}%
  \fontfamily{#3}\fontseries{#4}\fontshape{#5}%
  \selectfont}%
\fi\endgroup%
\begin{picture}(5565,3669)(1114,-4168)
\put(2579,-809){\makebox(0,0)[lb]{\smash{{\SetFigFont{7}{8.4}{\rmdefault}{\mddefault}{\updefault}$E_j^o$}}}}
\put(4450,-3233){\makebox(0,0)[lb]{\smash{{\SetFigFont{7}{8.4}{\rmdefault}{\mddefault}{\updefault}$E_\ell^o$}}}}
\put(5583,-861){\makebox(0,0)[lb]{\smash{{\SetFigFont{7}{8.4}{\rmdefault}{\mddefault}{\updefault}$E_k^o$}}}}
\put(3290,-940){\makebox(0,0)[lb]{\smash{{\SetFigFont{7}{8.4}{\rmdefault}{\mddefault}{\updefault}jump at random time}}}}
\put(5188,-2706){\makebox(0,0)[lb]{\smash{{\SetFigFont{7}{8.4}{\rmdefault}{\mddefault}{\updefault}jump at boundary}}}}
\put(2161,-3886){\makebox(0,0)[lb]{\smash{{\SetFigFont{7}{8.4}{\rmdefault}{\mddefault}{\updefault}jump at random time}}}}
\end{picture}%